\documentclass[11pt]{amsart}
\usepackage[utf8]{inputenc}

\usepackage{amsmath}
\usepackage{amssymb}
\usepackage{amsfonts}
\usepackage{amsrefs}
\usepackage{amsthm}
\usepackage{mathtools}
\usepackage{multirow}
\usepackage{enumitem}
\usepackage{tikz-cd}
\usepackage{soul}
\usepackage{mathrsfs}
\usepackage{bbm}
\usepackage{float}
\usepackage{abstract}
\usepackage{colortbl,dcolumn}
\usepackage{rotating}
\usepackage{afterpage}
\usepackage{subfig}
\usepackage{pifont}
\usepackage{relsize}
\usepackage{graphicx}
\usepackage{xcolor}
\usepackage{adjustbox}
\usepackage{tabularray}
\usepackage{geometry}
\usepackage{color}
\usepackage[colorlinks=true, pdfstartview=FitV, linkcolor=blue,citecolor=blue, urlcolor=blue]{hyperref}
\usepackage[all, knot]{xy}
\usepackage{tikz}
\usetikzlibrary{arrows.meta}
\usepackage{diagbox}
\usepackage{lscape}
\usepackage{array}
\usepackage{parskip}
\usepackage{setspace}
\numberwithin{equation}{section}
\newtheorem{Thm}{Theorem}[section]
\newtheorem{Prop}[Thm]{Proposition}
\newtheorem{Cor}[Thm]{Corollary}
\newtheorem{Lem}[Thm]{Lemma}
\theoremstyle{definition}
\newtheorem{defn}[Thm]{Definition}

\newtheorem{Rmk}[Thm]{Remark}

\newcommand{\aff}{\mathrm{aff}}
\newcommand{\Uq}{U_{q}(\mathfrak{g})}

\newcommand{\Fcal}{\mathcal{F}}

\newcommand{\Zcal}{\mathcal{Z}}
\newcommand{\Ycal}{\mathcal{Y}}

\newcommand{\pb}{\mathbf{p}}

\newcommand{\Vaff}{V^{\mathrm{aff}}}

\newcommand{\Baff}{B^{\mathrm{aff}}}
\newcommand{\Haff}{H^{\mathrm{aff}}}

\newcommand{\Zbb}{\mathbb{Z}}

\newcommand{\Pbar}{\overline{P}}
\newcommand{\et}{\Tilde{e}}
\newcommand{\ft}{\Tilde{f}}

\newcommand{\wpr}{\bigwedge^{r}\Vaff}

\newcommand{\sign}{\mathrm{sign}}
\newcommand{\presign}{\mathrm{pre{\hbox{-}}sign}}
\newcommand{\Z}{\mathbb{Z}}
\newcommand{\BB}{B}
\newcommand{\ep}{\varepsilon}
\newcommand{\wt}{{\rm wt}}

\title[Young wall models for the level 1 highest weight and Fock space crystals]
{Young wall models for the level 1 highest weight \\ and Fock space crystals of \boldmath $U_q(E_6^{(2)})$ and $U_q(F_4^{(1)})$}

\author[Shaolong Han]{Shaolong Han}
\address{Beijing International Center for Mathematical Research, Peking University, Beijing 100871, China}
\email{algebra@hrbeu.edu.cn}

\author[Yuanfeng Jin]{Yuanfeng Jin}
\address{Department of Mathematics, Yanbian University, Yanji, China}
\email{yfkim@ybu.edu.cn}

\author[Seok-Jin Kang]{Seok-Jin Kang}
\address{Korea Research Institute of Arts and Mathematics, Asan, Republic of Korea}
\email{soccerkang@hotmail.com}

\author[Duncan Laurie]{Duncan Laurie}
\address{Mathematical Institute, University of Oxford, Oxford, United Kingdom}
\email{duncan.laurie@maths.ox.ac.uk}



\colorlet{lightgreen}{white!55!green}
\colorlet{verylightgreen}{white!75!green}
\colorlet{lightorange}{white!65!orange}
\colorlet{verylightorange}{white!65!orange}
\colorlet{darkgreen}{black!10!green}
\colorlet{lightcyan}{white!75!cyan}

\def\cell#1{%
\ifdim#1pt=0pt\cellcolor{white}\else
\ifdim#1pt=1pt\cellcolor{verylightgreen}\else
\ifdim#1pt=2pt\cellcolor{lightorange}\else
\ifdim#1pt=-1pt\cellcolor{verylightred}\else
\cellcolor{lightcyan}\fi\fi\fi\fi
#1}

\definecolor{electricpurple}{rgb}{0.75, 0.0, 1.0}
\definecolor{shockingpink}{rgb}{0.99, 0.06, 0.75}
\definecolor{saddlebrown}{rgb}{0.55, 0.27, 0.07}
\definecolor{yellow(ncs)}{rgb}{1.0, 0.83, 0.0}

\colorlet{col0}{black}
\colorlet{col1}{red}
\colorlet{col2}{yellow(ncs)}
\colorlet{col3}{green!75!black}
\colorlet{col4}{electricpurple}
\colorlet{col5}{blue}
\colorlet{col6}{orange}
\colorlet{col7}{shockingpink}
\colorlet{col8}{saddlebrown!80!white}

\usepackage{color}
\usepackage{listings}
\usepackage{xspace}

\lstdefinelanguage{Sage}[]{Python}
{morekeywords={True,False,sage,singular},
sensitive=true}
\definecolor{dblackcolor}{rgb}{0.0,0.0,0.0}
\definecolor{dbluecolor}{rgb}{.01,.02,0.7}
\definecolor{dredcolor}{rgb}{0.8,0,0}
\definecolor{dgraycolor}{rgb}{0.30,0.3,0.30}

\begin{document}

\maketitle

\begin{abstract}
\vspace{4pt}
In this paper we construct Young wall models for the level $1$ highest weight and Fock space crystals of quantum affine algebras in types $E_6^{(2)}$ and $F_4^{(1)}$.
Our starting point in each case is a combinatorial realization for a certain level $1$ perfect crystal in terms of Young columns.
Then using energy functions and affine energy functions we define the notions of reduced and proper Young walls, which model the highest weight and Fock space crystals respectively.

\end{abstract}

\vspace{10pt}

\begin{spacing}{1.25}
\tableofcontents	
	\end{spacing}

\newpage

\section*{Introduction}

The theory of crystal bases introduced by Kashiwara \cites{Kas90,Kas91} provides a powerful tool for studying the representation theory of quantum groups.
These crystal bases can be seen as the $q = 0$ limits of global bases, and possess a host of combinatorial features that reflect the internal structures of the corresponding representations.
On the other hand, Lusztig developed the theory of canonical bases using a more geometric approach \cites{Lus90,Lus91}.

The combinatorial aspects of crystal bases can often transport abstract algebraic problems about quantum groups and their representations into far more tractable settings.
For example, the characters and tensor decompositions of integrable modules can be expressed using explicit descriptions of their crystals.
Constructing concrete realizations of crystal bases is therefore an important research topic in this area.

Young tableau and Young wall models provide a particularly intuitive and easy-to-operate class of such realizations.
In all non-exceptional finite types, Kashiwara and Nakashima \cite{KN94} described the crystal bases of all finite dimensional $\Uq$-modules in terms of generalised Young tableaux.
Moreover, Kang and Misra \cite{KM94} provided a similar construction in type $G_{2}$.

In order to approach the case of quantum affine algebras, Kang et al \cites{KMN1,KMN2} used the theory of perfect crystals to construct a \textit{path realization} for the irreducible highest weight crystals $B(\lambda)$, known as the Kyoto path model.
Here, paths are infinite sequences of elements inside some (finite) perfect crystal which stabilise to a certain ground-state path.

Using this path realization, Kang \cite{Kang03} obtained Young wall models for the level $1$ highest weight crystals of quantum affine algebras in types $A_{n}^{(1)}$, $A_{2n-1}^{(2)}$, $A_{2n}^{(2)}$, $B_{n}^{(1)}$, $D_{n}^{(1)}$ and $D_{n+1}^{(2)}$.
The remaining non-exceptional affine type $C_n^{(1)}$ was subsequently addressed in work of Hong-Kang-Lee \cite{HKL04}.
Furthermore, this programme was later generalised to arbitrary level by Kang and Lee in \cites{KL06,KL09}.

As for the exceptional affine types, Fan-Han-Kang-Shin \cite{FHKS23} built Young wall models for the level $1$ irreducible highest weight crystals in types $D_4^{(3)}$ and $G_2^{(1)}$ using the perfect crystals of \cite{KMOY07} and \cites{JM11,MMO10,Yam98} respectively.
Moreover, Laurie \cite{Lau23} recently constructed such models in types $E_{6}^{(1)}$, $E_{7}^{(1)}$ and $E_{8}^{(1)}$.

Fock space representations $\Fcal(\lambda)$ of quantum affine algebras -- initially studied in type $A$ by Kashiwara, Miwa and Stern \cites{KMS95,Ste95} -- were developed in the general setting by Kashiwara-Miwa-Petersen-Yung \cite{KMPY96}.
They are obtained by affinizing a finite dimensional module $V$ (which must satisfy certain assumptions), and then taking the semi-infinite limit of the $q$-exterior powers along a vacuum vector.
In particular, Kashiwara \cite{Kas02} showed that it suffices for $V$ to be a \textit{good module}.

It was shown that the crystal basis $B(\Fcal(\lambda))$ of the Fock space can be expressed in terms of the crystal basis $B$ for $V$ and its energy function $H$.
Since $B$ must be a perfect crystal, it is natural to seek realizations for $B(\Fcal(\lambda))$ in terms of
Young walls, similar to those for $B(\lambda)$.

Such models for the level $1$ Fock space crystals were constructed by
Kang-Kwon, Kim-Shin, Misra-Miwa, and Premat
in types $A_{n}^{(1)}$, $A_{2n-1}^{(2)}$, $A_{2n}^{(2)}$, $B_{n}^{(1)}$, $D_{n}^{(1)}$, $D_{n+1}^{(2)}$, $C_{n}^{(1)}$ (cf. \cites{MM90,KK03,KK04,KK08,Pre04,KS04}) and by Laurie in types $E_{6}^{(1)}$, $E_{7}^{(1)}$, $E_{8}^{(1)}$ (cf. \cite{Lau23}).
Moreover, via a similar treatment to \cite{Lau23}*{§5} or Section \ref{Sec:Fock space Young wall realization} of this paper, the existing Young wall models for $B(\lambda)$ of level $1$ in types  $D_{4}^{(3)}$ and $G_{2}^{(1)}$ can be readily adapted to provide Young wall models for $B(\Fcal(\lambda))$.

The aim of this paper is to construct Young wall models for the level $1$ irreducible highest weight and Fock space crystals in the final remaining affine types $E_{6}^{(2)}$ and $F_{4}^{(1)}$, thereby completing this section of the story.

In many types, the level $1$ perfect crystal used as a starting point is the crystal basis of some level $0$ fundamental representation associated to a non-zero minuscule vertex of the affine Dynkin diagram.
For example, this is the case in \cite{Kang03} and also for types $E_{6}^{(1)}$ and $E_{7}^{(1)}$ in \cite{Lau23}.
However, in other types no such vertex exists and so we must look elsewhere for an appropriate crystal.
In this paper, as for type $E_{8}^{(1)}$ in \cite{Lau23}, we shall use a uniform construction due to Benkart-Frenkel-Kang-Lee \cite{BFKL06}.

Our strategy is as follows.
We first realize the level $1$ perfect crystal in each type in terms of equivalence classes of Young columns.
These Young columns are certain stackings of colored blocks within a relevant \textit{Young column pattern}, which is obtained by splitting an infinite vertical strip of cuboids into building blocks of various shapes.
Young columns are equivalent if they can be obtained from one another via vertical shift or $180^{\circ}$ rotation around the vertical axis, and their crystal structure is defined simply in terms of adding and removing blocks.

Following the descriptions of $B(\lambda)$ and $B(\mathcal{F}(\lambda))$ in terms of our level $1$ perfect crystal, we arrange infinitely many Young column patterns right to left in order to create a \textit{Young wall pattern}.
Similarly, we can form the ground-state wall and thus introduce the notion of a Young wall stacked inside the Young wall pattern.

With the help of energy functions on our level $1$ perfect crystals, we define the sets of \textit{reduced} Young walls and \textit{proper} Young walls.
Using the tensor product rule for Young walls these can each be endowed with the structure of an affine crystal, with the Kashiwara operators given by adding and removing colored blocks.
Moreover we show that the resulting crystals provide combinatorial Young wall realizations for the crystal bases $B(\lambda)$ and $B(\mathcal{F}(\lambda))$ respectively.
Furthermore, in order to better understand and determine these reduced and proper Young walls, we prove certain structural results including a particular \textit{right block property} in each case.

This paper is organized as follows.
Section \ref{Sec:Crystals} recalls some necessary preliminaries regarding crystals for quantum affine algebras, in particular concerning perfect crystals, energy functions, and the path realization of highest weight crystals.
We also briefly summarise the construction of the Fock space representations as semi-infinite limits of $q$-exterior powers, as well as the corresponding characterization of their crystal bases.

Section \ref{Sec: Perfect crystal of E and F} describes the level $1$ perfect crystal of Benkart-Frenkel-Kang-Lee in types $E_6^{(2)}$ and $F_4^{(1)}$, and provides illustrations of the crystal graph in each case.
Furthermore, the energy function values -- which we calculate using SageMath \cite{SageMath} -- are displayed in Appendix \ref{appe:HandH'values}.

In Section \ref{Sec:Young column models} we build the Young column patterns for types $E_6^{(2)}$ and $F_4^{(1)}$ out of colored building blocks, and define the set of Young columns stacked inside each pattern.
We then realize our level $1$ perfect crystals in terms of equivalence classes of these Young columns, and provide complete lists of all the classes.
Appendix \ref{realization of B and B'} contains the crystal graph of each Young column realization.

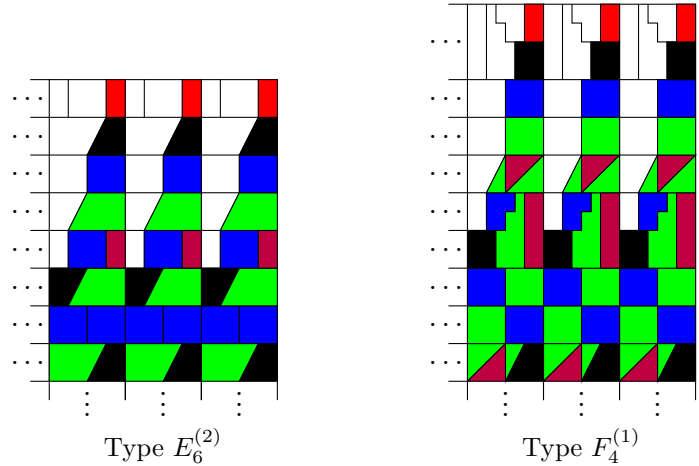
We begin Section \ref{Sec:highest weight Young wall realization} by first constructing the Young wall patterns and ground-state walls for types $E_6^{(2)}$ and $F_4^{(1)}$, which allows us to define the notion of a Young wall.
We then use the energy functions on our level $1$ perfect crystals to define the reduced Young walls via a combinatorial condition on adjacent columns.
We present structural results for these walls, and show that the set of reduced Young walls has an affine crystal structure.
Moreover we prove that this crystal provides a Young wall realization for the level $1$ highest weight crystal in each type.

Similarly, in Section \ref{Sec:Fock space Young wall realization} we give the definition of a proper Young wall in terms of a combinatorial energy function condition.
We endow the set of proper Young walls with the structure of an affine crystal, and show that it is isomorphic to the level $1$ Fock space crystal in each type.
In Section \ref{proper structure} we study the structure of these proper Young walls in more detail, for example determining when adjacent columns satisfy the right block property.
Appendix \ref{top crystal} displays the top part of our Young wall models for $B(\lambda)$ and $B(\mathcal{F}(\lambda))$ in each type.

\noindent
{\bfseries  Acknowledgements.}
S.-J. Kang would like to thank Professor Zhaobing Fan for his hospitality during a stay at Harbin Engineering University in November 2023.
D. Laurie was financially supported by the Engineering and Physical Sciences Research Council [grant number EP/T517811/1].

\section{Crystals of quantum affine algebras}\label{Sec:Crystals}
\subsection{Perfect crystals}
Let us fix some basic notations.
\begin{itemize}
\item[\textendash] $I = \{0,1,...,n\}$: index set.
\item[\textendash] $A=(a_{ij})_{i,j\in I}$: affine Cartan matrix.
\item[\textendash] $D={\rm diag}\{s_i\in \mathbb Z_{>0}\mid i\in I\}$: diagonal matrix such that $DA$ is symmetric.
\item[\textendash] $P^{\vee} =(\oplus_{i\in I}\mathbb Z h_i) \oplus
\Z d$: dual weight lattice.
\item[\textendash] $\mathfrak{h}=\mathbb C \otimes_{\mathbb Z}
P^{\vee}$: Cartan subalgebra.
\item[\textendash] $P=\{\lambda \in
\mathfrak{h}^* \mid \lambda(P^{\vee})\subset \Z\}$: affine weight lattice.
\item[\textendash] $\Pi^\vee=\{h_i\mid i\in I\}\subset P^\vee$: the set of simple coroots.
\item[\textendash] $\Pi=\{\alpha_i\mid i\in I\}\subset P$: the set of simple roots.
\item[\textendash] $\mathfrak g$: affine Lie algebra associated to the Cartan datum $(A, P^{\vee}, P, \Pi^\vee, \Pi)$.
\item[\textendash] $\delta$, $c$, $\Lambda_i\ (i\in I)$: null root, canonical central element, fundamental weights.
\item[\textendash] $P^{+}$: the set of affine dominant integral weights.
\item[\textendash] $\Pbar=\oplus_{i\in I}\mathbb Z \Lambda_i$: the set of classical weights.
\item[\textendash] $\Pbar^{+}$: the set of classical dominant integral weights.
\item[\textendash] $l=\lambda(c)$: the level of an affine or classical dominant integral weight $\lambda$.
\end{itemize}

We furthermore define
\begin{align*}
q_i = q^{s_i},\quad \left[n\right]_i=\dfrac{q_i^n-q_i^{-n}}{q_i-q_i^{-1}},\quad
\left[m\right]_i!=\left[m\right]_i\left[m-1\right]_i \dots
\left[1\right]_i,\quad \left[0\right]_i!=1,
\end{align*}
for each $i\in I$, $n\in\mathbb Z$ and $m\in \mathbb Z_{> 0}$.

The {\it quantum affine algebra} $U_q(\mathfrak{g})$ is the unital associative algebra over $\mathbb Q(q)$
generated by $e_i, f_i$ $(i\in I)$ and $q^h$ $( h \in P^{\vee})$,
subject to the following defining relations.
\begin{enumerate}
	
	\item $q^0 = 1,\ q^h q^{h^{\prime}} = q^{h + h^{\prime}}$  for all  $h,h^{\prime} \in P^{\vee}$,
	
	\item  $q^h e_i q^{-h} = q^{\alpha_i(h)}e_i,\ q^h f_i q^{-h} = q^{-\alpha_i(h)}f_i$  for all  $h \in P^{\vee}$,
	
	\item  $e_i f_j - f_j e_i = \delta_{ij} \dfrac{K_i - K_i^{-1}}{q_i - q_i^{-1}}$  for all  $i,j \in I$,
	
	\item $\displaystyle\sum_{k=0}^{1-a_{ij}}(-1)^k e_{i}^{(1-a_{ij}-k)} e_j e_{i}^{(k)} = 0 $ whenever $i \ne j$,

	\item $\displaystyle\sum_{k=0}^{1-a_{ij}}(-1)^k f_{i}^{(1-a_{ij}-k)} f_j f_{i}^{(k)} = 0 $  whenever $i \ne j$,
\end{enumerate}
where $e_i^{(k)}= \frac{e_i^k}{\left[k\right]_{i}!}$,
$f_i^{(k)} = \frac{f_i^k}{\left[ k \right]_{i}!}$ and $K_i = q_i^{h_i}$ for each $i\in I$ and $k\in\mathbb{Z}_{\geq 0}$.

We denote by $U_q'(\mathfrak{g})$ the subalgebra of
$U_q(\mathfrak{g})$ generated by $e_i,f_i,K_i^{\pm1}$ $(i \in I)$.

Both $U_q(\mathfrak{g})$ and $U_q'(\mathfrak{g})$ have a coproduct $\Delta$ given by
\begin{align*} \label{coproduct}
    \Delta(q^{h}) = q^{h}\otimes q^{h}, \,\,
    \Delta(e_{i}) = e_{i}\otimes K_{i}^{-1} + 1\otimes e_{i}, \,\,
    \Delta(f_{i}) = f_{i}\otimes 1 + K_{i}\otimes f_{i},
\end{align*}
which shall be used for the construction of Fock space representations.

\begin{Rmk}
    Our coproduct $\Delta$ is as in \cites{KMN1,Kas02,Lau23}, whereas that of \cite{KMPY96} is obtained by exchanging the tensor factors -- see \cite{KMPY96}*{\S 2.2} for more details.
\end{Rmk}

\begin{defn} \label{def:crystal}
	An {\it affine crystal} (resp. a {\it classical crystal}) is
	a set $\BB$ together with maps $\wt : \BB \rightarrow P$
	(resp. $\wt: \BB \rightarrow \Pbar$), $\tilde{e}_i,
	\tilde{f}_i: \BB \rightarrow \BB \cup \{ 0 \}$ and $\varepsilon_i,
	\varphi_i : \BB \rightarrow \mathbb{Z} \cup \{-\infty\}$  $(i \in I)
	$ satisfying the following conditions.
	\begin{enumerate}
		\item  $\varphi_{i}(b) = \varepsilon_{i}(b) + \langle h_i, \wt
		(b) \rangle$ \ for all  $ i \in I$,
		
		\item  $ \wt(\tilde {e}_i b) = \wt (b) + \alpha_i$ if  $\tilde{e}_i b \in \BB$,
		
		\item  $ \wt(\tilde{f}_i b) = \wt(b) - \alpha_i$ if  $\tilde{f}_ib \in \BB$,
		
		\item  $\varepsilon_i(\tilde{e} _ib) = \varepsilon_i(b) - 1$,
		$\varphi_i(\tilde{e}_i b) = \varphi_i(b) + 1$ if $\tilde{e}_ib \in
		\BB$,

		\item  $\varepsilon_i(\tilde{f} _ib) = \varepsilon_i(b) + 1$,
		$\varphi_i(\tilde{f}_i b) = \varphi_i(b) - 1$ if  $\tilde{f}_ib \in
		\BB$,
		
		\item  $\tilde{f}_i b = b'$ if and only if $b =
		\tilde{e}_i b'$ for all $b,b'\in \BB$ and $i \in I $,
		
		\item  If $\varphi_i(b) = -\infty$  for  $b\in \BB$, then
		$\tilde{e}_i b = \tilde{f}_i b = 0$.
		
	\end{enumerate}
\end{defn}

\begin{defn}
	A {\it crystal morphism} $\Psi : B \rightarrow B'$ between two affine or classical crystals is a map
	$\Psi : B \cup \{0\} \rightarrow B' \cup \{0\}$ such that
	\begin{enumerate}
		\item  $\Psi(0)=0$,
		\item if $b \in B$ and $\Psi(b) \in B'$, then $\wt(\Psi(b))=\wt(b)$, $\varepsilon_i(\Psi(b))=\varepsilon_i(b)$ and $\varphi_i(\Psi(b))=\varphi_i(b)$ for all $i \in I$,
		\item if $b, b' \in B$, $\Psi(b), \Psi(b') \in B'$ and $\tilde{f}_i b=b'$, then $\tilde{f}_i \Psi(b)=\Psi(b')$ and $\Psi(b)=\tilde{e}_i \Psi(b')$ for all $i \in I$.
	\end{enumerate}
	Moreover $\Psi$ is an {\it isomorphism} if $\Psi : B\cup\lbrace 0\rbrace \rightarrow B'\cup\lbrace 0\rbrace$ is a bijection.
\end{defn}

We define the {\it tensor
	product} $B \otimes B'$ of two affine or classical crystals $B$ and $B'$ to be the set $B \times B' $
with a crystal structure given by
\begin{equation} \label{tensor product of crystals}
	\begin{aligned}
		\tilde {e}_i(b \otimes b')&= 
\begin{cases} \tilde {e}_i b \otimes b' &\text{if \;$\varphi_i(b)\ge \varepsilon_i(b')$}, \\[-.5ex] 
			b\otimes \tilde {e}_i b' & \text{if\; $\varphi_i(b) < \varepsilon_i(b')$}, 
\end{cases} \\[-.5ex] 
		\tilde {f}_i(b \otimes b')&= 
\begin{cases} \tilde {f}_i b \otimes b'
			&\text{if \; $\varphi_i(b) > \varepsilon_i(b')$}, \\[-.5ex] 
			b\otimes \tilde {f}_i b' & \text{if \; $\varphi_i(b) \le
				\varepsilon_i(b')$},
\end{cases}\\	
		\wt(b \otimes b')&= \wt(b)+\wt(b'),\\
		\ep_{i}(b \otimes b')&= \max(\ep_{i}(b), \ep_i(b') -
		\langle h_i, \wt(b) \rangle ),\\
		\varphi_{i}(b \otimes b')& = \max(\varphi_{i}(b'), \varphi_i(b)
		+\langle h_i, \wt(b')\rangle ).
	\end{aligned}
\end{equation}

We next introduce the notion of a perfect crystal as developed in \cites{KMN1,KMN2} by Kang et al.
For each element $b$ of a classical crystal $B$, define associated classical weights
\begin{eqnarray*}
	\varepsilon(b) = \displaystyle\sum_{i \in I} \varepsilon_i(b)\Lambda_i, &
	\varphi(b) = \displaystyle\sum_{i \in I} \varphi_i(b)\Lambda_i.
\end{eqnarray*}

\begin{defn}\label{def:perfect crystal}
	Let $l$ be a positive integer. A classical crystal $\BB$ is
	called a \textit{perfect crystal of level $l$} if
	
	\begin{enumerate}
		\item  there exists an irreducible finite dimensional $U'_q (\mathfrak{g})$-module with a crystal basis isomorphic to $\BB$,
		
		\item  $\BB \otimes \BB$ is connected,
		
		\item  there exists a classical weight $ \lambda_0 \in \Pbar$
		such that 
		$$\displaystyle \wt(\BB) \subset \lambda_0 + \sum_{i \ne 0}
		\mathbb{Z}_{\le0} \: \alpha_i, \quad
		\#(\BB_{\lambda_0})=1,$$ 
		where
		$\BB_{\lambda_0}=\{ b \in \BB ~ | ~ \wt(b)=\lambda_0 \}$,
		
		\item  $\varepsilon(b)(c)
		\ge l $ for all $b \in \BB$,
		
		\item  for any $ \lambda \in \Pbar ^{+}$ with $\lambda(c)= l$ there
		exist  unique $ b^{\lambda}\in B $ and $b_{\lambda} \in B$ with
		$\varepsilon(b^{\lambda})  = \varphi(b_{\lambda})= \lambda.$
		
	\end{enumerate}
\end{defn}

The vectors $b^\lambda$ and $b_\lambda$ are called the \textit{minimal vectors}.

\subsection{Path realization of highest weight crystals} \label{section:highest weight crystals}
The importance of these perfect crystals is demonstrated by the following results.

\begin{Thm} [{\cite{KMN1}}]
	Let $\BB$ be a perfect crystal of level $l \in
	\mathbb{Z}_{>0}$. For any $\lambda \in \Pbar^+$ with
	$\lambda(c)= l$ there exists a unique classical crystal
	isomorphism
	\begin{equation*}
		\begin{aligned}
			\Psi :\BB(\lambda) \stackrel{\sim}{\longrightarrow}
			\BB(\ep(b_\lambda)) \otimes \BB &
			&\text{given by}& &  u_\lambda  \longmapsto  u_{\ep(b_\lambda)} \otimes b_\lambda ,
		\end{aligned}
	\end{equation*}
	where $u_\lambda$ is the highest weight vector in $\BB(\lambda)$ and $b_\lambda$ is the unique vector in $\BB$ such that $\varphi(b_{\lambda})=\lambda$.
\end{Thm}

Let 
$$
\lambda_0 =\lambda, \quad \lambda_{k+1}=\ep(b_{\lambda_k}), \quad b_0=b_{\lambda_0}, \quad b_{k+1}=b_{\lambda_{k+1}}
$$
for all $k\in\mathbb{Z}_{\geq 0}$.
Repeatedly applying the theorem above produces a sequence of crystal isomorphisms
\begin{equation*}
	\begin{array}{ccccccc} \BB(\lambda) &
		\stackrel{\sim}{\longrightarrow} & \BB(\lambda_1) \otimes \BB &
		\stackrel{\sim}{\longrightarrow} & \BB(\lambda_2)\otimes \BB \otimes
		\BB & \stackrel{\sim}{\longrightarrow} & \cdots
		\\
		u_\lambda & \longmapsto & u_{\lambda_1} \otimes b_0 & \longmapsto &
		u_{\lambda_2} \otimes b_1 \otimes b_0 & \longmapsto & \cdots.
\end{array} \end{equation*}
In this process, we obtain an infinite sequence $\mathbf{p}_\lambda =(b_k)^{\infty} _{k=0} \in \BB^{\otimes \infty}$
called the \textit{ground-state path of weight} $\lambda$.
The set
$$
\mathcal{P}(\lambda):=\{\mathbf{p}=(p_{k})^{\infty}_{k=0} \in \BB^{\otimes \infty} \;|\; p_{k} \in \BB,\ p_{k}=b_k \ \text{for all} \ k \gg 0  \}
$$
of $\lambda$-\textit{paths} is endowed with the structure of a classical crystal as follows.
If $p_{k}=b_{k}$ for all $k\geq r$ then let
\begin{align} \label{crystal structure on paths}
\begin{split}
    &\wt(\pb) = \lambda_{r} + \wt(\pb'), \\
    &\et_{i}\pb = \dots\otimes p_{r+1}\otimes\et_{i}(p_{r}\otimes\dots\otimes p_{0}), \\
    &\ft_{i}\pb = \dots\otimes p_{r+1}\otimes\ft_{i}(p_{r}\otimes\dots\otimes p_{0}), \\
    &\varepsilon_{i}(\pb) = \max(\varepsilon_{i}(\pb') - \varphi_{i}(b_{r}),0), \\
    &\varphi_{i}(\pb) = \varphi_{i}(\pb') + \max(\varphi_{i}(b_{r}) - \varepsilon_{i}(\pb'),0),
\end{split}
\end{align}
where $\pb' = p_{r-1}\otimes\dots\otimes p_{0}$.
The following result gives the {\it path realization} of the irreducible highest weight crystal $\BB(\lambda)$.

\begin{Prop}[{\cite{KMN1}}]\label{prop:path realization}
	There exists an isomorphism of classical crystals
	\begin{equation*}
			\Psi_{\lambda} :\BB(\lambda) \stackrel{\sim}{\longrightarrow}
			\mathcal{P}({\lambda})\ \
			\text{given by}\ \ u_\lambda \longmapsto\mathbf{p}_\lambda,
	\end{equation*}
	where $u_{\lambda}$ is the highest weight vector in $B(\lambda)$.
	\label{path_realization}
\end{Prop}

\subsection{Energy functions and affinizations}
\begin{defn} \label{def:energy function}
	An \textit{energy function} on an affine or classical crystal $B$ is a map $H:B\otimes B\to\mathbb Z$ satisfying
	\begin{equation*}
		H(\tilde{f}_i(b_1\otimes b_2))=
		\begin{cases}
			H(b_1\otimes b_2), &\text{if } i\neq 0,\\
			H(b_1\otimes b_2)-1, &\text{if } i=0,\ \varphi_0(b_1)>\varepsilon_0(b_2),\\
			H(b_1\otimes b_2)+1, &\text{if } i=0,\ \varphi_0(b_1)\leq\varepsilon_0(b_2),
		\end{cases}
	\end{equation*}	
	for each $i\in I$ and $b_1\otimes b_2\in B\otimes B$ with $\tilde{f}_i(b_1\otimes b_2)\in B\otimes B$.
\end{defn}

Such a function is therefore determined uniquely up to constant shift on each connected component of $B\otimes B$.

The existence of an energy function $H$ for every perfect crystal $B$ was proven in \cite{KMN2}, allowing us to upgrade the path realization of $B(\lambda)$ above to an isomorphism of \emph{affine} crystals.
In particular, we replace the weight function in (\ref{crystal structure on paths}) with
\begin{align} \label{affine weights on paths}
    \wt(\pb) = \lambda_{r} + \wt(\pb') + \delta\,\sum_{k=0}^{\infty} (k+1)(H(p_{k+1}\otimes p_{k})\! -\! H(b_{k+1}\otimes b_{k})).
\end{align}

Any classical crystal has an associated affine crystal defined as follows.

\begin{defn}
The affinization of a classical crystal $B$ is the set $ B^{\mathrm{aff}}:=\{b(n) \mid b \in B, n \in \Z \}$, with an affine crystal structure given by
\begin{equation*}
	\begin{aligned}
		& \tilde{e}_{i}(b(n)) = (\tilde{e}_{i} b)(n-\delta_{i0}), \quad 
		\tilde{f}_{i} (b(n)) = (\tilde{f}_{i} b)(n+\delta_{i0}), \\ 
&\varepsilon_i(b(n))=\varepsilon_i(b),\quad
\varphi_i(b(n))=\varphi_i(b),\quad \wt(b(n))=\wt(b)-n\delta.
	\end{aligned}
\end{equation*} 	
\end{defn}

It is clear that for any morphism (resp. isomorphism) $\Psi : B \rightarrow B'$ of classical crystals, $\Psi^{\mathrm{aff}}(b(n)) := (\Psi(b))(n)$ defines a morphism (resp. isomorphism) $\Psi^{\mathrm{aff}} : \Baff \rightarrow B'^{\mathrm{aff}}$ between their affinizations.

Furthermore, given an energy function $H$ on a classical crystal $B$, we can define a corresponding energy function $\Haff$ on its affinization.

\begin{defn}\label{affine energy function}
The \emph{affine energy function} $H^{\mathrm{aff}}:B^{\mathrm{aff}}\otimes B^{\mathrm{aff}}\to \mathbb Z$ is given by
	\begin{equation*}
		H^{\mathrm{aff}}(a(m)\otimes b(n))=H(a\otimes b)+m-n
	\end{equation*} 
	for each $a,b\in B$ and $m,n\in\mathbb Z$.
\end{defn}

\begin{Thm}[\cite{FHKS23}*{Lemma 3.12}]\label{constant}
	The affine energy function $H^{\mathrm{aff}}$ is constant on each connected component
	of $B^{\mathrm{aff}}\otimes B^{\mathrm{aff}}$.
\end{Thm}

\begin{Rmk} \label{energy function difference remark}
    It is important to note that while our definitions of energy functions and affine energy functions match those of references such as \cites{BFKL06,FHKS23,KMN1,KMN2,KMPY96}, they are equal to \emph{minus} those of \cites{Kas02,Lau23}.
\end{Rmk}

\subsection{Fock space crystals} \label{Fock space preliminaries}
Here we shall briefly outline the construction due to Kashiwara-Miwa-Petersen-Yung \cite{KMPY96} of the Fock space representations $\Fcal(\lambda)$ for  quantum affine algebra $U_q(\mathfrak{g})$, together with a description of their crystal bases.

We start with a finite dimensional $U'_q(\mathfrak{g})$-module $V$ satisfying certain assumptions.
In particular, it was shown by Kashiwara \cite{Kas02} that we can take $V$ to be a \textit{good module}, which means that it has a simple crystal basis $B$, a bar involution, and a global basis -- see \cites{Kas02,Lau23} for more details.
Let us further assume that $B$ is a perfect crystal of level $l$.

Consider the affinization $\Vaff$ as a representation of $U_q(\mathfrak{g})$, and define a submodule
\begin{align*}
    N = U_q(\mathfrak{g})[z^{\pm 1}\otimes z^{\pm 1},z\otimes 1 + 1\otimes z](u\otimes u)
\end{align*}
of the tensor square $\Vaff \otimes \Vaff$, which is independent of a choice of extremal vector $u\in\Vaff$.

The $q$-exterior power $\wpr$ is then the quotient of $(\Vaff)^{\otimes r}$ by
\begin{align*}
    N_{r} = \sum_{k=0}^{r-2} (\Vaff)^{\otimes k}\otimes N \otimes (\Vaff)^{\otimes(r-k-2)},
\end{align*}
which can be thought of as a deformation of the ordinary exterior power since $N = \ker(R-1)$, where $R$ is the action of $R$-matrix on $\Vaff\otimes\Vaff$.

Fix some weight $\lambda\in\Pbar^{+}$ of level $l$.
With $\mathbf{p}_\lambda = (b_k)^{\infty} _{k=0} \in \BB^{\otimes \infty}$ as in Section \ref{section:highest weight crystals}, we let $m_{k}\in\mathbb{Z}$ be such that
\begin{align*}
    \Haff(b_{k+1}(m_{k+1})\otimes b_{k}(m_{k})) = 1
\end{align*}
for all $k\geq 0$.
Then $(b_k(m_{k}))^{\infty}_{k=0} \in (\Baff)^{\otimes \infty}$ is called the \textit{ground-state sequence} for the Fock space, and the corresponding element
$\cdots\wedge v^{\circ}_{2}\wedge v^{\circ}_{1}\wedge v^{\circ}_{0}$
in $(\Vaff)^{\otimes \infty}$ is the \textit{vacuum vector}.

The \textit{Fock space} $\Fcal(\lambda)$ is defined to be the semi-infinite limit $\lim_{r\rightarrow\infty} \wpr$ along this vacuum vector, and can be naturally endowed with a $U_q(\mathfrak{g})$-module structure.
(Technically, we quotient by a small subspace to ensure that certain sums converge and the action is well-defined.)

Any element of $\Fcal(\lambda)$ can then be written as a linear combination of infinite wedges $\cdots\wedge v_{2}\wedge v_{1}\wedge v_{0}$ with $v_{k} = v^{\circ}_{k}$ for $k\gg 0$.

We say that a sequence $(p_{k}(n_{k}))^{\infty}_{k=0}$ in $\Baff$ is \textit{normally ordered} if
\begin{align} \label{normally ordered condition}
    \Haff(p_{k+1}(n_{k+1})\otimes p_{k}(n_{k})) < 2
\end{align}
for all $k\geq 0$.
Note the different condition compared to \cite{Kas02} in order to account for Remark \ref{energy function difference remark}.

\begin{Thm}[\cite{KMPY96}]
    The set of normally ordered sequences $(p_{k}(n_{k}))^{\infty}_{k=0}$ in $\Baff$ with $p_{k}(n_{k}) = b_{k}(m_{k})$ for $k\gg 0$, endowed with the structure of an affine crystal via (\ref{crystal structure on paths}) and (\ref{affine weights on paths}), forms a crystal basis $B(\Fcal(\lambda))$ for the Fock space $\Fcal(\lambda)$.
\end{Thm}

\begin{Prop} \label{component of ground-state sequence proposition}
    The connected component of the ground-state sequence in $B(\Fcal(\lambda))$ is a copy of $B(\lambda)$ consisting of the sequences $(p_{k}(n_{k}))^{\infty}_{k=0}$ with all $\Haff(p_{k+1}(n_{k+1})\otimes p_{k}(n_{k})) = 1$.
\end{Prop}

\section{Level \texorpdfstring{$1$}{1} perfect crystals of \texorpdfstring{$U_q(E_6^{(2)})$}{Uq(E6(2))} and \texorpdfstring{$U_q(F_4^{(1)})$}{Uq(F4(1))}} \label{Sec: Perfect crystal of E and F}
Let us recall the affine Cartan data of types $E_6^{(2)}$ and $F_4^{(1)}$.
Take $I=\{0,1,2,3,4\}$ to be the index set, and denote the sets of simple coroots, simple roots and fundamental weights respectively by
$$
\{h_0, h_1, h_2, h_3, h_4\},\quad \{\alpha_0, \alpha_1, \alpha_2, \alpha_3, \alpha_4\},\quad \{\Lambda_0, \Lambda_1, \Lambda_2, \Lambda_3, \Lambda_4\}.
$$

The Cartan matrices are equal to
\begin{equation*}
E_6^{(2)}:\ \left(
\begin{array}{rrrrr}
	2  & -1 & 0 & 0 & 0  \\
	-1 & 2  & -1 & 0 & 0 \\
	0  & -1 & 2 & -2 & 0 \\
	0  & 0 & -1 & 2 & -1 \\
	0  & 0 & 0 & -1 & 2 		
\end{array}
\right) 
\end{equation*}
\begin{equation*}
	F_4^{(1)}:\ \left(
	\begin{array}{rrrrr}
		2  & -1 & 0 & 0 & 0  \\
		-1 & 2  & -1 & 0 & 0 \\
		0  & -1 & 2 & -1 & 0 \\
		0  & 0 & -2 & 2 & -1 \\
		0  & 0 & 0 & -1 & 2 		
	\end{array}
	\right) 
\end{equation*}
and their associated Dynkin diagrams are illustrated as
\[
E_6^{(2)}:\ 	
\xymatrix@R=.0ex{
	\circ\ar@{-}[r]&\circ\ar@{-}[r]&\circ&\circ\ar@2{->}[l]\ar@{-}[r]&\circ	\\
	0&1&2&3&4
}
\]

\[
F_4^{(1)}:\ 	
\xymatrix@R=.0ex{
\circ\ar@{-}[r]&\circ\ar@{-}[r]&\circ&\circ\ar@2{<-}[l]\ar@{-}[r]&\circ	\\
0&1&2&3&4
}
\]

The null root $\delta$ and canonical central element $c$ are 
given by
\begin{align*}
&E_6^{(2)}:\ \delta=\alpha_0+2\alpha_1+3\alpha_2+2\alpha_3+\alpha_4, \quad c=h_0+2h_1+3h_2+4h_3+2h_4,\\
&F_4^{(1)}:\ 
\delta=\alpha_0+2\alpha_1+3\alpha_2+4\alpha_3+2\alpha_4, \quad c=h_0+2h_1+3h_2+2h_3+h_4.
\end{align*}

Our level $1$ perfect crystals of $U_q(E_6^{(2)})$ and $U_q(F_4^{(1)})$ shall come from a uniform construction due to Benkart--Frenkel--Kang--Lee \cite{BFKL06}.

In particular, let $\Phi^+$ and $\Phi^-=-\Phi^+$ be the sets of positive and negative roots in types $F_{4}^{t}$ and $F_{4}$ respectively, which are obtained from the corresponding affine Cartan data by removing the $0$ vertex from the affine Dynkin diagram.

Let $\theta=\delta-\alpha_0 \in \Phi^{+}$ be the highest short root for $F_{4}^{t}$ and highest root for $F_{4}$, and define $B(0)=\{\emptyset\}$ and $B(\theta)=\{x_{\pm\alpha}\mid \alpha\in \Phi^+\}\cup\{r_i\mid i=1,2,3,4\}$.
The set $B(\theta)\sqcup B(0)$ can be turned into a crystal graph by adding the following arrows, whereby it gains the structure of a classical crystal.
\begin{equation}\label{eq:crys} 
\begin{array}{cccc}
		\vspace{3mm}
		&(i \neq 0)   &\qquad x_\alpha\ \stackrel{i}{\longrightarrow} \ x_\beta \ \Longleftrightarrow \ \alpha-\alpha_i = \beta & \qquad (\alpha,\beta \in
		\Phi^+\cup \Phi^-) \\
		\vspace{3mm}
		& &\qquad  x_{\alpha_i} \ \stackrel{i}{\longrightarrow} \ r_i \ \stackrel{i}{\longrightarrow} \ x_{-\alpha_i} & \qquad (\alpha_i \in \Phi^+) \\
			\vspace{3mm}
		&(i=0)  &\qquad x_\alpha \ \stackrel{0}{\longrightarrow} \ x_{\beta} \ \Longleftrightarrow \ \alpha+\theta = \beta &
		\qquad (\alpha,\beta\in\Phi^{+}\sqcup\Phi^{-}, \alpha,\beta \neq \pm \theta) \\
			\vspace{3mm}
		&  &\qquad x_{-\theta} \ \stackrel{0}{\longrightarrow} \ \emptyset \ \stackrel{0}{\longrightarrow}\ x_{\theta} & \end{array}
\end{equation}

\begin{Thm}[\cite{BFKL06}*{Theorem 3.1}]\label{thm:level 1 perfect crystal} 
    Equipped with the crystal structure in \eqref{eq:crys}, the set $B(\theta) \sqcup  B(0)$ becomes a level $1$ perfect crystal of the quantum affine algebra $U_q(E_6^{(2)})$ or $U_q(F_4^{(1)})$. 
\end{Thm}

We can therefore use $B(\theta) \sqcup  B(0)$ for the path realization of the level $1$ irreducible highest weight crystals $B(\lambda)$.
Furthermore, the following result from \cite{Lau23} shows that it can also be used to construct the level $1$ Fock space crystals $B(\Fcal(\lambda))$ as in Section \ref{Fock space preliminaries}.

\begin{Prop}[\cite{Lau23}*{Proposition 3.4}] \label{BFKL crystals are good}
    The level $1$ perfect crystal of Benkart-Frenkel-Kang-Lee is the crystal basis of a good $U'_q(\mathfrak{g})$-module in all affine types.
\end{Prop}

In each case, we shall write $(a_1a_2a_3a_4)$ as shorthand for any $a_1\alpha_1+a_2\alpha_2+a_3\alpha_3+a_4\alpha_4\in\Phi^{+}$.
The positive roots of $F_{4}^{t}$ are then given by
\begin{align*}
	&(1000),\quad (0100),\quad (1100),\quad (0110),\quad (1110),\quad (0111),\\
	&(1111),\quad (1210),\quad (1211),\quad (1221),\quad (1321),\quad (2321),
\end{align*}
while those of $F_{4}$ are given by
\begin{align*}
&(1000),\quad (0100),\quad (0010),\quad (0001),\quad (1100),\quad (0110),\quad (0011),\quad (1110),\\
&(0120),\quad (0111),\quad (1120),\quad (1111),\quad (0121),\quad (1220),\quad (1121),\quad (0122),\\
&(1221),\quad (1122),\quad (1231),\quad (1222),\quad (1232),\quad (1242),\quad (1342),\quad (2342).
\end{align*}
Furthermore, we shall represent each negative root $-(a_1a_2a_3a_4)\in\Phi^{-}$ by $\overline{(a_1a_2a_3a_4)}$.

In types $E_6^{(2)}$ and $F_4^{(1)}$ the only dominant integral weight $\lambda\in\Pbar^{+}$ of level $l=\lambda(c)=1$ is $\Lambda_{0}$.
It is clear from Definition \ref{def:perfect crystal} that in either case the corresponding minimal vectors are $b_{\lambda} = b^{\lambda} = \emptyset$.

Throughout the remainder of this paper, we shall color the arrows in our crystal graphs and the blocks in our Young columns and Young walls according to their label $i\in I$, in particular
\begin{align} \label{color conventions}
    {\color{red} 0\mathrm{~is~red}};\quad
    {\color{black} 1\mathrm{~is~black}};\quad
    {\color{blue} 2\mathrm{~is~blue}};\quad
    {\color{green} 3\mathrm{~is~green}};\quad
    {\color{purple} 4\mathrm{~is~purple}}.
\end{align}

Figure \ref{crystal graph of B} and Figure \ref{crystal graph of B'} contain the crystal graphs of the level $1$ perfect crystals $B$ of $U_q(E_6^{(2)})$ and $B'$ of $U_q(F_4^{(1)})$ from Theorem \ref{thm:level 1 perfect crystal}.

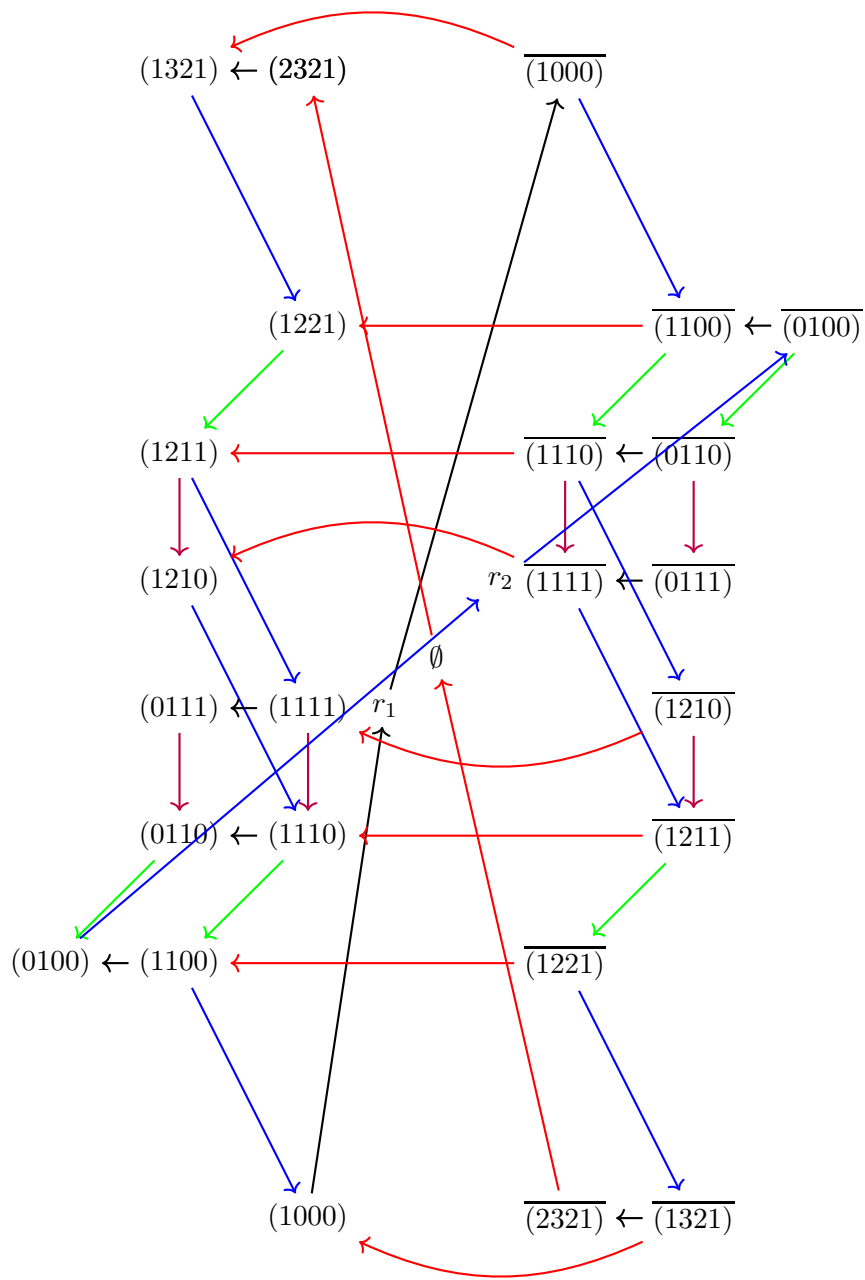
\begin{figure}[H] 
	\centering
	\begin{tikzpicture}[scale=1.69]
		\node (0100) at (0,2) {$(0100)$};	
		\node (1100) at (1,2) {$(1100)$};
		\node (0110) at (1,3) {$(0110)$};
		\node (0111) at (1,4) {$(0111)$};
		\node (1210) at (1,5) {$(1210)$};	    
		\node (1211) at (1,6) {$(1211)$};	    
		\node (1321) at (1,9) {$(1321)$};	    
		\node (1000) at (2,0) {$(1000)$};	   
		\node (1110) at (2,3) {$(1110)$};	    
		\node (1111) at (2,4) {$(1111)$};
		\node (1221) at (2,7) {$(1221)$};	    
		\node (2321) at (2,9) {$(2321)$};	    
		\node (2321) at (2,9) {$(2321)$};	    
		\node (r1) at (2.6,4) {$r_1$};    
		\node (empty) at (3,4.4) {$\emptyset$};
		\node (r2) at (3.5,5) {$r_2$};		
		\node (-2321) at (4,0) {$\overline{(2321)}$};		
		\node (-1221) at (4,2) {$\overline{(1221)}$};		
		\node (-1111) at (4,5) {$\overline{(1111)}$};		
		\node (-1110) at (4,6) {$\overline{(1110)}$};		
		\node (-1000) at (4,9) {$\overline{(1000)}$};		
		\node (-1321) at (5,0) {$\overline{(1321)}$};		
		\node (-1211) at (5,3) {$\overline{(1211)}$};		
		\node (-1210) at (5,4) {$\overline{(1210)}$};	
		\node (-0111) at (5,5) {$\overline{(0111)}$};		
		\node (-0110) at (5,6) {$\overline{(0110)}$};		
		\node (-1100) at (5,7) {$\overline{(1100)}$};		
		\node (-0100) at (6,7) {$\overline{(0100)}$};	
		%%%%%%%%%%%%%%%
		\draw[thick,black,->]  (2321) -- (1321);
		\draw[thick,blue,->]  (1321) -- (1221);
		\draw[thick,green,->]  (1221) -- (1211);
		\draw[thick,purple,->]  (1211) -- (1210);
		\draw[thick,blue,->]  (1211) -- (1111);
		\draw[thick,blue,->]  (1210) -- (1110);
		\draw[thick,black,->]  (1111) -- (0111);
		\draw[thick,purple,->]  (0111) -- (0110);
		\draw[thick,purple,->]  (1111) -- (1110);	
		\draw[thick,black,->]  (1110) -- (0110);
		\draw[thick,green,->]  (0110) -- (0100);	
		\draw[thick,green,->]  (1110) -- (1100);
		\draw[thick,black,->]  (1100) -- (0100);
		\draw[thick,blue,->]  (1100) -- (1000);
		%%%%%%%%%%%%%%
		\draw[thick,blue,->]  (-1000) -- (-1100);
		\draw[thick,black,->]  (-0100) -- (-1100);
		\draw[thick,green,->]  (-1100) -- (-1110);
		\draw[thick,green,->]  (-0100) -- (-0110);
		\draw[thick,black,->]  (-0110) -- (-1110);
		\draw[thick,purple,->]  (-1110) -- (-1111);
		\draw[thick,purple,->]  (-0110) -- (-0111);
		\draw[thick,black,->]  (-0111) -- (-1111);
		\draw[thick,blue,->]  (-1110) -- (-1210);
		\draw[thick,blue,->]  (-1111) -- (-1211);
		\draw[thick,purple,->]  (-1210) -- (-1211);
		\draw[thick,green,->]  (-1211) -- (-1221);
		\draw[thick,blue,->]  (-1221) -- (-1321);
		\draw[thick,black,->]  (-1321) -- (-2321);
		%%%%%%%%%%%%
		\draw[thick,black,->]  (1000) -- (r1);
		\draw[thick,black,->]  (r1) -- (-1000);
		
		\draw[thick,red,->]  (-2321) -- (empty);
		\draw[thick,red,->]  (empty) -- (2321);
		\draw[thick,blue,->]  (0100) -- (r2);
		\draw[thick,blue,->]  (r2) -- (-0100);
		%%%%%%%%%%%%
		\draw[thick,red,->,bend right=25]  (-1000) to node [swap] {} (1321);
		\draw[thick,red,->]  (-1100) -- (1221);
		\draw[thick,red,->]  (-1110) -- (1211);
		\draw[thick,red,->,bend right=25]  (-1111) to node [swap] {} (1210);
		\draw[thick,red,->,bend left=25]  (-1210) to node [swap] {} (1111);
		\draw[thick,red,->]  (-1211) -- (1110);
		\draw[thick,red,->]  (-1221) -- (1100);
		\draw[thick,red,->,bend left=25]  (-1321) to node [swap] {} (1000);

	\end{tikzpicture}
	\caption{The crystal graph of $B$}\label{crystal graph of B}
\end{figure}

\begin{figure}[H]
	\centering
	\begin{tikzpicture}[scale=1.7]
		\node (2342) at (0,0) {$(2342)$};	
		\node (1342) at (0,-1) {$(1342)$};
		\node (1242) at (-1,-2) {$(1242)$};
		\node (1232) at (-0.5,-3) {$(1232)$};
		\node (1231) at (-1.5,-3) {$(1231)$};
		\node (1231) at (-1.5,-3) {$(1231)$};
		\node (1222) at (0,-4) {$(1222)$};
		\node (1221) at (-1,-4) {$(1221)$};
		\node (1220) at (-2,-4) {$(1220)$};
		\node (1122) at (-1,-5) {$(1122)$};
		\node (1121) at (-2,-5) {$(1121)$};
		\node (1120) at (-3,-5) {$(1120)$};
		\node (0122) at (-1,-6) {$(0122)$};
		\node (0121) at (-2,-6) {$(0121)$};
		\node (0120) at (-3,-6) {$(0120)$};
		\node (1111) at (-1,-7) {$(1111)$};
		\node (1110) at (-2,-7) {$(1110)$};
		\node (0111) at (-1,-8) {$(0111)$};
		\node (0110) at (-2,-8) {$(0110)$};
		\node (1100) at (-1,-9) {$(1100)$};
		\node (0011) at (-2,-9) {$(0011)$};
		\node (0010) at (-3,-9) {$(0010)$};
		\node (0100) at (-1,-10) {$(0100)$};		
		\node (1000) at (-2,-10) {$(1000)$};		
		\node (0001) at (-1,-11) {$(0001)$};		
%%%%%%%%%%%%%%%%%%%%%%%%%%%%
		\node (-0001) at (4,0) {$\overline{(0001)}$};
		\node (-0100) at (4,-1) {$\overline{(0100)}$};		
		\node (-1000) at (5,-1) {$\overline{(1000)}$};
		\node (-1100) at (4,-2) {$\overline{(1100)}$};	
		\node (-0011) at (5,-2) {$\overline{(0011)}$};
		\node (-0010) at (6,-2) {$\overline{(0010)}$};
		\node (-0111) at (4,-3) {$\overline{(0111)}$};
		\node (-0110) at (5,-3) {$\overline{(0110)}$};
		\node (-1111) at (4,-4) {$\overline{(1111)}$};	
		\node (-1110) at (5,-4) {$\overline{(1110)}$};	
		\node (-0122) at (4,-5) {$\overline{(0122)}$};    
		\node (-0121) at (5,-5) {$\overline{(0121)}$};  
		\node (-0120) at (6,-5) {$\overline{(0120)}$};  	
		\node (-1122) at (4,-6) {$\overline{(1122)}$};  	
		\node (-1121) at (5,-6) {$\overline{(1121)}$};  
		\node (-1120) at (6,-6) {$\overline{(1120)}$};  
		\node (-1222) at (3,-7) {$\overline{(1222)}$};  
		\node (-1221) at (4,-7) {$\overline{(1221)}$};	
		\node (-1220) at (5,-7) {$\overline{(1220)}$};	
		\node (-1232) at (3.5,-8) {$\overline{(1232)}$};	
		\node (-1231) at (4.5,-8) {$\overline{(1231)}$};	
		\node (-1242) at (4,-9) {$\overline{(1242)}$};	
		\node (-1342) at (3,-10) {$\overline{(1342)}$};
		\node (-2342) at (3,-11) {$\overline{(2342)}$};
%%%%%%%%%%%%%%%%%%%%%%%%%%%%
		\node (h1) at (0.5,-5.5) {$r_1$};		
		\node (h2) at (1,-5.5) {$r_2$};	
		\node (empty) at (1.5,-5.5) {$\emptyset$};	
		\node (h3) at (2,-5.5) {$r_3$};
		\node (h4) at (2.5,-5.5) {$r_4$};	
%%%%%%%%%%%%%%%%%%%%%%%%%%%%			
\draw[thick,black,->] (2342)--(1342);	
\draw[thick,blue,->] (1342)--(1242);	
\draw[thick,green,->] (1242)--(1232);
\draw[thick,purple,->] (1232)--(1231);
\draw[thick,green,->] (1232)--(1222);
\draw[thick,green,->] (1231)--(1221);
\draw[thick,purple,->] (1222)--(1221);
\draw[thick,purple,->] (1221)--(1220);
\draw[thick,blue,->] (1222)--(1122);
\draw[thick,blue,->] (1221)--(1121);
\draw[thick,blue,->] (1220)--(1120);
\draw[thick,purple,->] (1122)--(1121);
\draw[thick,purple,->] (1121)--(1120);
\draw[thick,black,->] (1122)--(0122);
\draw[thick,black,->] (1121)--(0121);
\draw[thick,black,->] (1120)--(0120);
\draw[thick,purple,->] (0122)--(0121);
\draw[thick,purple,->] (0121)--(0120);
\draw[thick,green,->] (1121)--(1111);
\draw[thick,green,->] (1120)--(1110);
\draw[thick,green,->] (0121)--(0111);
\draw[thick,green,->] (0120)--(0110);
\draw[thick,purple,->] (1111)--(1110);
\draw[thick,black,->] (1110)--(0110);
\draw[thick,black,->] (1111)--(0111);
\draw[thick,purple,->] (0111)--(0110);
\draw[thick,green,->] (1110)--(1100);
\draw[thick,green,->] (0110)--(0100);
\draw[thick,blue,->] (0111)--(0011);
\draw[thick,blue,->] (0110)--(0010);
\draw[thick,purple,->] (0011)--(0010);
\draw[thick,black,->] (1100)--(0100);
\draw[thick,blue,->] (1100)--(1000);
\draw[thick,green,->] (0011)--(0001);
%%%%%%%%%%%%%%%%%%%%%%%%%%%%
\draw[thick,green,->] (-0001)--(-0011);
\draw[thick,black,->] (-0100)--(-1100);
\draw[thick,green,->] (-0100)--(-0110);
\draw[thick,blue,->] (-1000)--(-1100);
\draw[thick,blue,->] (-0011)--(-0111);
\draw[thick,blue,->] (-0010)--(-0110);
\draw[thick,purple,->] (-0010)--(-0011);
\draw[thick,purple,->] (-0110)--(-0111);
\draw[thick,black,->] (-0110)--(-1110);
\draw[thick,black,->] (-0111)--(-1111);
\draw[thick,purple,->] (-1110)--(-1111);
\draw[thick,green,->] (-0110)--(-0120);
\draw[thick,green,->] (-1100)--(-1110);
\draw[thick,green,->] (-1110)--(-1120);
\draw[thick,green,->] (-0111)--(-0121);
\draw[thick,green,->] (-1111)--(-1121);
\draw[thick,purple,->] (-0120)--(-0121);
\draw[thick,purple,->] (-0121)--(-0122);
\draw[thick,black,->] (-0120)--(-1120);
\draw[thick,black,->] (-0121)--(-1121);
\draw[thick,black,->] (-0122)--(-1122);
\draw[thick,purple,->] (-1120)--(-1121);
\draw[thick,purple,->] (-1121)--(-1122);
\draw[thick,blue,->] (-1120)--(-1220);	
\draw[thick,blue,->] (-1121)--(-1221);	
\draw[thick,blue,->] (-1122)--(-1222);	
\draw[thick,purple,->] (-1220)--(-1221);		
\draw[thick,purple,->] (-1221)--(-1222);	
\draw[thick,green,->] (-1221)--(-1231);	
\draw[thick,green,->] (-1222)--(-1232);		
\draw[thick,purple,->] (-1231)--(-1232);			
\draw[thick,green,->] (-1232)--(-1242);	
\draw[thick,blue,->] (-1242)--(-1342);		
\draw[thick,black,->] (-1342)--(-2342);	
%%%%%%%%%%%%%%%%%%%%%%%	
\draw[thick,red,->,bend right=20] (-1000) to node [swap] {}(1342);		
\draw[thick,red,->] (-1100) -- (1242);	
\draw[thick,red,->,bend right=10] (-1110) to node [swap] {} (1232);	
\draw[thick,red,->,bend left=10] (-1111) to node [swap] {} (1231);
\draw[thick,red,->] (-1120) -- (1222);
\draw[thick,red,->] (-1121) -- (1221);
\draw[thick,red,->] (-1122) -- (1220);
\draw[thick,red,->] (-1220) -- (1122);
\draw[thick,red,->] (-1221) -- (1121);
\draw[thick,red,->] (-1222) -- (1120);	
\draw[thick,red,->,bend right=10] (-1231) to node [swap] {} (1111);	
\draw[thick,red,->,bend left=10] (-1232) to node [swap] {} (1110);
\draw[thick,red,->] (-1242) -- (1100);	
\draw[thick,red,->,bend left=20] (-1342) to node [swap] {}(1000);		
%%%%%%%%%%%%%%%%%%%%%%%	
\draw[thick,black,->,bend right=7] (1000) to node [swap] {}(h1);
\draw[thick,black,->,bend left=11] (h1) to node [swap] {}(-1000);
\draw[thick,blue,->] (0100)--(h2);
\draw[thick,blue,->] (h2)--(-0100);
\draw[thick,green,->,bend right=10] (0010)to node [swap] {}(h3);
\draw[thick,green,->,bend left=10] (h3) to node [swap] {}(-0010);
\draw[thick,purple,->,bend right=10] (0001) to node [swap] {} (h4);
\draw[thick,purple,->,bend left=10] (h4) to node [swap] {} (-0001);
\draw[thick,red,->] (-2342) to node [swap] {} (empty);
\draw[thick,red,->] (empty) to node [swap] {} (2342);
\end{tikzpicture}
\caption{The crystal graph of $B'$}\label{perfect graph B'} \label{crystal graph of B'}
\end{figure}
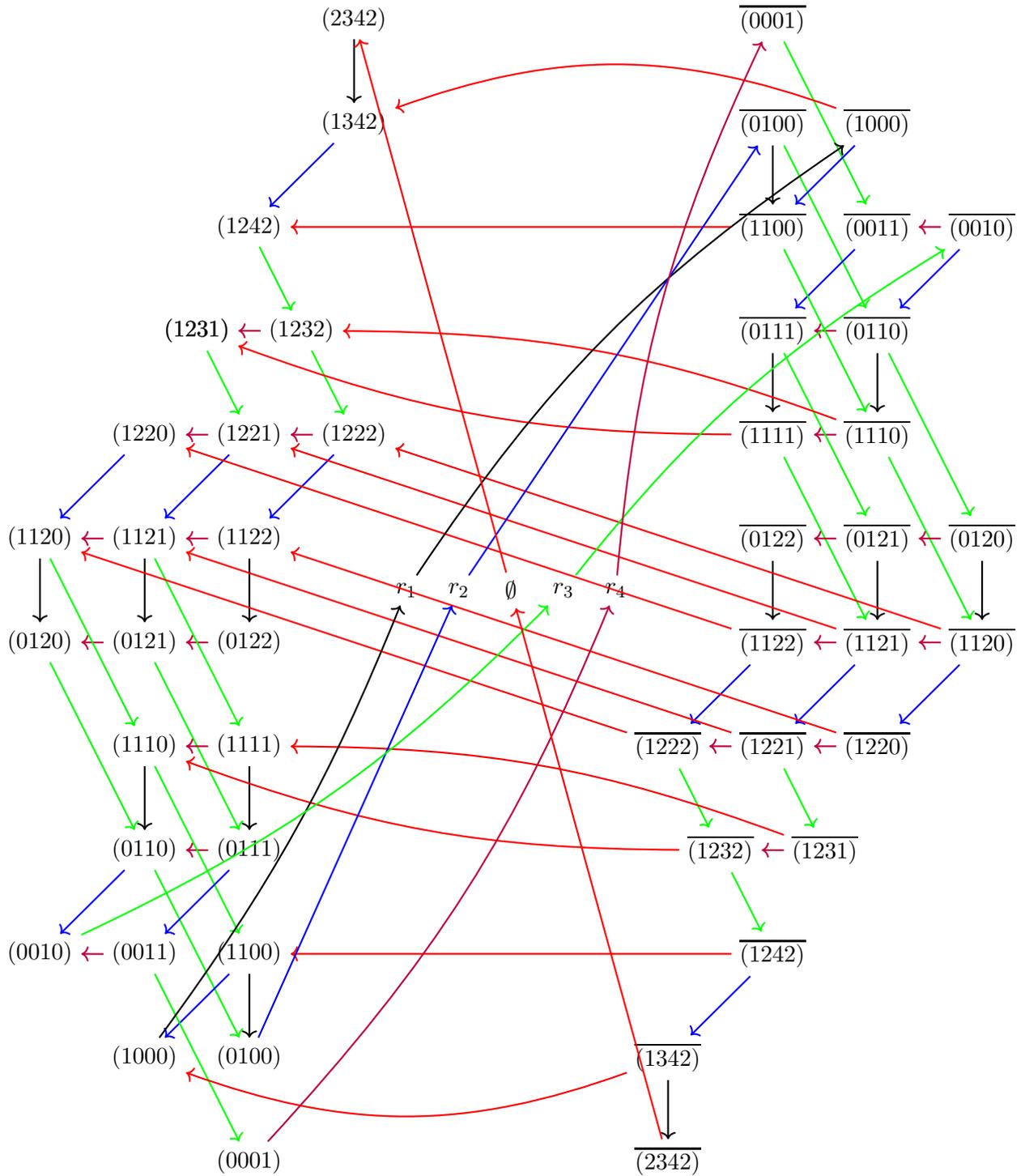

Appendix \ref{appe:HandH'values} contains the values of the energy functions $H:B\otimes B\to\mathbb Z$ and $H':B'\otimes B'\to\mathbb Z$, which we calculated using SageMath \cite{SageMath}.

In particular, both $B$ and $B'$ are isomorphic to the Kirillov-Reshetikhin crystal $B^{1,1}$ in their respective types $X_{n}^{(r)}$, and so the following code outputs a list of energy function values.

\begin{lstlisting}
sage: K = crystals.kirillov_reshetikhin.LSPaths(['X',n,r],1)
sage: K.digraph().edges()
sage: H = K.local_energy_function(K)
sage: K2 = crystals.TensorProduct(K,K)
sage: for b in K2:
          print("({},{}) {}".format(b[1],b[0],2-H(b)))
\end{lstlisting}

\begin{Rmk}
    \begin{enumerate}
        \item The factors {\tt b[0]} and {\tt b[1]} of {\tt b} are reversed in our final line of code since by default SageMath uses a reversed tensor crystal structure.
        \item We output {\tt 2-H(b)} since (after accounting for the reversed tensor structure) SageMath uses the energy function definition of \cites{Kas02,Lau23}, which is in particular \textit{minus} that of Definition \ref{def:energy function} up to constant shift.
        Moreover we normalise so that $H(\emptyset\otimes\emptyset) = 0$.
    \end{enumerate}
\end{Rmk}

\section{Young column models for level \texorpdfstring{$1$}{1} perfect crystals of \texorpdfstring{$U_q(E_6^{(2)})$}{Uq(E6(2))} and \texorpdfstring{$U_q(F_4^{(1)})$}{Uq(F4(1))}} \label{Sec:Young column models}

\subsection{Building blocks}
In this subsection we introduce the building blocks required to construct our Young column and Young wall models.

The dimensions of a cuboid are written as $* \times * \times *$, representing its $\mathrm{width} \times \mathrm{thickness} \times \mathrm{height}$ measurements.

In type $E_6^{(2)}$ the $2\times 1\times 1$ cuboid is split in four different ways via a collection of vertical cuts, as shown in Figure \ref{cutting process}.
This process produces building blocks of four different shapes -- the unit block, 1/2-unit block, 5/4-unit block and 3/4-unit block -- which are named according to their volumes.
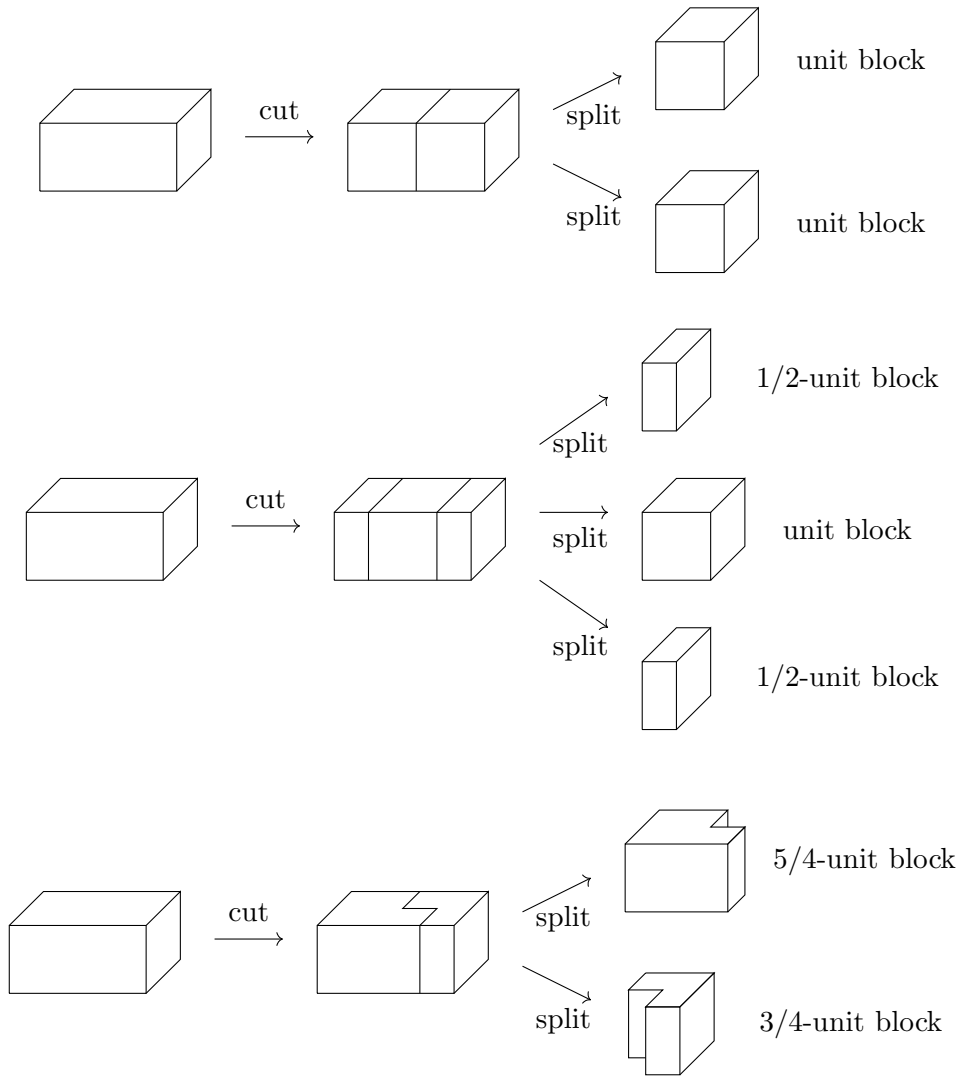
\begin{figure}[H]
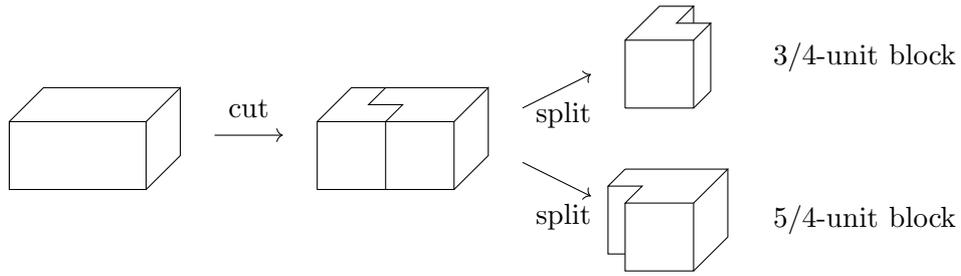

	\begin{center}

\end{center}
	\caption{The cutting process for the $2\times 1\times 1$ cuboid} \label{cutting process}
\end{figure}

In order to more easily display these three-dimensional blocks, we use the following planar diagrams.
\begin{figure}[H]
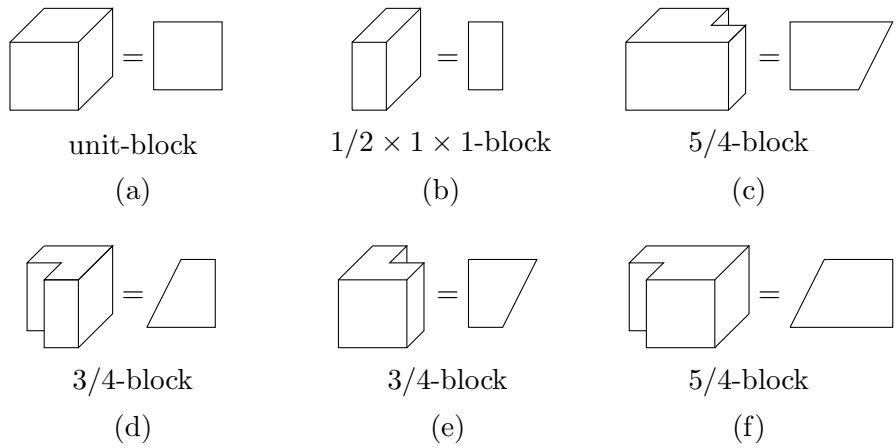

\begin{center}
%
\end{center}
\caption{Planar notation for the building blocks in type $E_{6}^{(2)}$} \label{planar notation}
\end{figure}

Note that blocks (c) and (f) (resp. (d) and (e)) above can be obtained from one another by $180^{\circ}$ rotation around the vertical axis.

Coloring our building blocks and planar diagrams according to (\ref{color conventions}), only the following are required to construct our Young column and Young wall models in type $E_{6}^{(2)}$.
\begin{figure}[H]
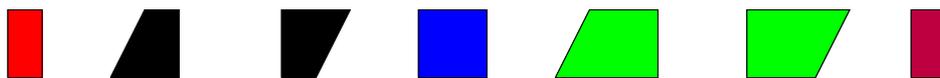

	\centering
	%
	\vspace{3pt}
\caption{Colored blocks in type $E_{6}^{(2)}$} \label{colored E62}	
\end{figure}

Figure \ref{F4(1) cuboid building blocks} contains the cuboid building blocks in type $F_4^{(1)}$ together with their planar diagrams.
\begin{figure}[H]
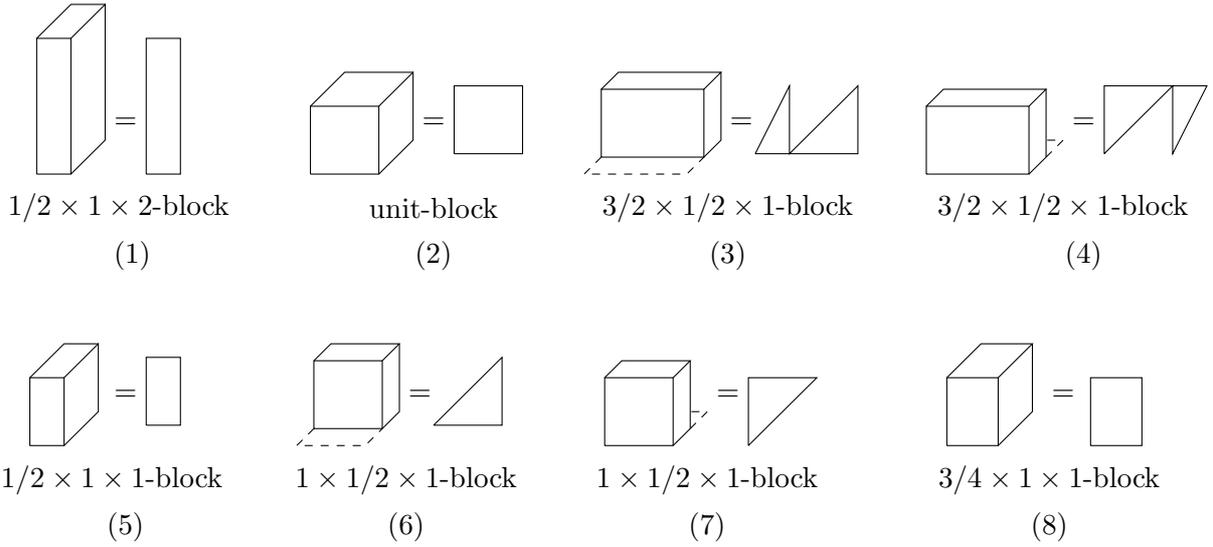

	\centering
	%
\caption{Planar notation for the cuboid building blocks in type $F_{4}^{(1)}$} \label{F4(1) cuboid building blocks}
\end{figure}

Moreover we shall also require the following additional collection of blocks.
\begin{figure}[H]
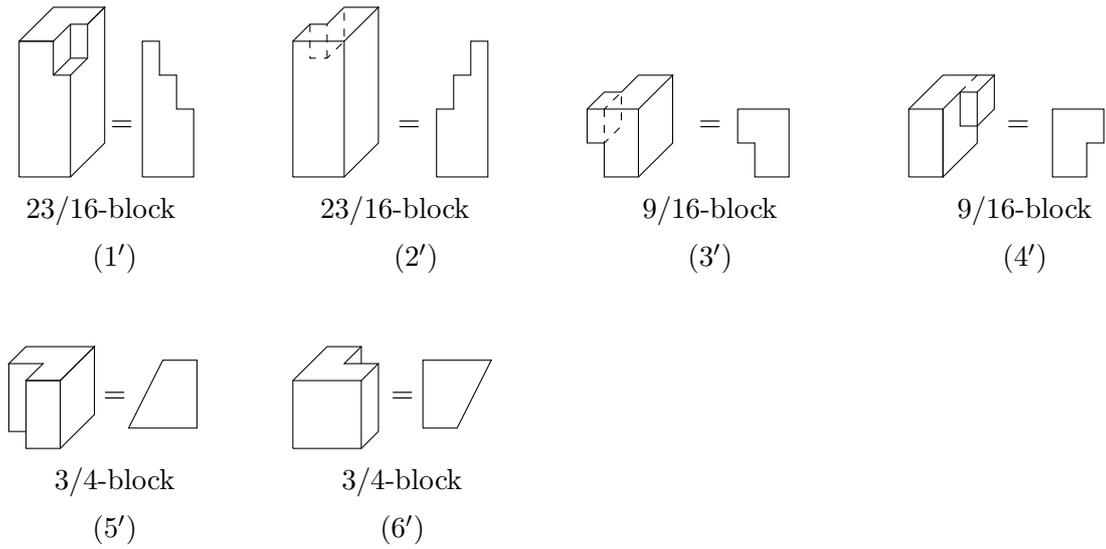

	\centering
	%
\caption{Planar notation for the non-cuboid building blocks in type $F_{4}^{(1)}$} \label{F4(1) non-cuboid building blocks}
\end{figure}

In particular, our Young column and Young wall models in type $F_{4}^{(1)}$ are constructed using the following colored blocks.
\begin{figure}[H]
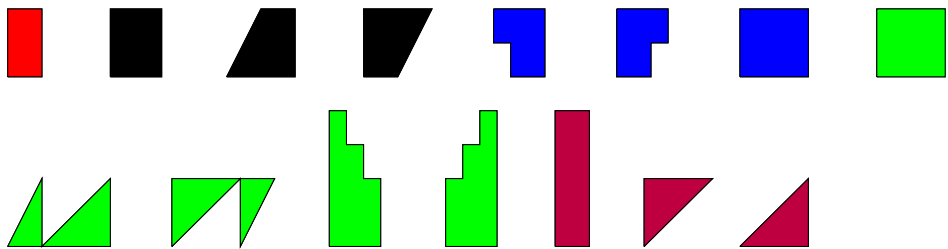

	\centering
	%

\caption{Colored blocks in type $F_{4}^{(1)}$} \label{colored F41}	
\end{figure}

We remark that blocks $(3)$, $(6)$, $(1')$, $(3')$ and $(5')$ in Figures \ref{F4(1) cuboid building blocks} and \ref{F4(1) non-cuboid building blocks} can be obtained from blocks $(4)$, $(7)$, $(2')$, $(4')$ and $(6')$ respectively by $180^{\circ}$ rotation around the vertical axis.

\subsection{Young columns}
Figure \ref{Young column patterns} contains the Young column patterns for types $E_{6}^{(2)}$ and $F_{4}^{(1)}$.

\vspace{-7pt}

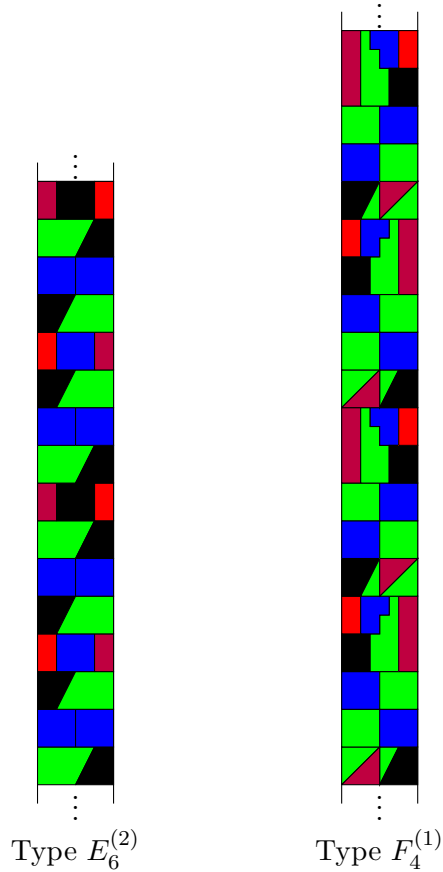
\begin{figure}[H]
	\centering
	\begin{tikzpicture}[scale=0.5]
%%%%%%%%%%%E_6^{(2)}%%%%%%%%%%%%%%%		
	\draw (0,0)--(0,8)--(2,8)--(2,0);
	\draw (0,7)--(2,7);
	\draw (0,6)--(2,6);	
	\draw (0,5)--(2,5);		
	\draw (0,4)--(2,4);		
	\draw (0,3)--(2,3);		
	\draw (0,2)--(2,2);		
	\draw (0,1)--(2,1);		
\draw (0.5,3)--(0.5,4)--(1,5)--(1,6)--(1.5,7)--(1.5,8);	
\draw 	(0.5,7)--(0.5,8);
\draw (0.5,2)--(1,3);
\draw (1,1)--(1,2);
\draw (1,0)--(1.5,1);
\draw (1.5,3)--(1.5,4);	
\draw (0,0)--(0,-0.5);	
\draw (2,0)--(2,-0.5);
\node at (1,-0.4) {$\vdots$};

\filldraw[fill=green,fill opacity=0.8]  (0,0)--(1,0)--(1.5,1)--(0,1)--(0,0);		
\filldraw[fill=black,fill opacity=0.8]  (1,0)--(2,0)--(2,1)--(1.5,1)--(1,0);	
\filldraw[fill=blue,fill opacity=0.8]  (0,1)--(1,1)--(1,2)--(0,2)--(0,1);	
\filldraw[fill=blue,fill opacity=0.8]  (1,1)--(2,1)--(2,2)--(1,2)--(1,1);		
\filldraw[fill=black,fill opacity=0.8]  (0,2)--(0.5,2)--(1,3)--(0,3)--(0,2);
\filldraw[fill=green,fill opacity=0.8]  (0.5,2)--(2,2)--(2,3)--(1,3)--(0.5,2);
\filldraw[fill=red,fill opacity=0.8]  (0,3)--(0,4)--(0.5,4)--(0.5,3)--(0,3);
\filldraw[fill=blue,fill opacity=0.8]  (0.5,3)--(0.5,4)--(1.5,4)--(1.5,3)--(0.5,3);
\filldraw[fill=purple,fill opacity=0.8]  (1.5,3)--(1.5,4)--(2,4)--(2,3)--(1.5,3);
\filldraw[fill=black,fill opacity=0.8]  (0,4)--(0,5)--(1,5)--(0.5,4)--(0,4);
\filldraw[fill=green,fill opacity=0.8]  (0.5,4)--(2,4)--(2,5)--(1,5)--(0.5,4);
\filldraw[fill=blue,fill opacity=0.8]  (0,5)--(0,6)--(1,6)--(1,5)--(0,5);
\filldraw[fill=blue,fill opacity=0.8]  (1,5)--(1,6)--(2,6)--(2,5)--(1,5);
\filldraw[fill=green,fill opacity=0.8]  (0,6)--(0,7)--(1.5,7)--(1,6)--(0,6);
\filldraw[fill=black,fill opacity=0.8]  (1,6)--(2,6)--(2,7)--(1.5,7)--(1,6);
\filldraw[fill=purple,fill opacity=0.8]  (0,7)--(0,8)--(0.5,8)--(0.5,7)--(0,7);
\filldraw[fill=black,fill opacity=0.8]  (0.5,7)--(0.5,8)--(1.5,8)--(1.5,7)--(0.5,7);
\filldraw[fill=red,fill opacity=0.8]  (1.5,7)--(1.5,8)--(2,8)--(2,7)--(1.5,7);
\node at (1,-1.7) {Type $E_6^{(2)}$};

%%%%%%%%%%%%%%%%%%%%%%%%%%%%
\begin{scope}[shift={(0,8)}]
	\draw (0,8)--(0,8.5);
	\draw (2,8)--(2,8.5);
\node at (1,8.6) {$\vdots$};			
	\draw (0,0)--(0,8)--(2,8)--(2,0);
\draw (0,7)--(2,7);
\draw (0,6)--(2,6);	
\draw (0,5)--(2,5);		
\draw (0,4)--(2,4);		
\draw (0,3)--(2,3);		
\draw (0,2)--(2,2);		
\draw (0,1)--(2,1);		
\draw (0.5,3)--(0.5,4)--(1,5)--(1,6)--(1.5,7)--(1.5,8);	
\draw 	(0.5,7)--(0.5,8);
\draw (0.5,2)--(1,3);
\draw (1,1)--(1,2);
\draw (1,0)--(1.5,1);
\draw (1.5,3)--(1.5,4);	
\filldraw[fill=green,fill opacity=0.8]  (0,0)--(1,0)--(1.5,1)--(0,1)--(0,0);		
\filldraw[fill=black,fill opacity=0.8]  (1,0)--(2,0)--(2,1)--(1.5,1)--(1,0);	
\filldraw[fill=blue,fill opacity=0.8]  (0,1)--(1,1)--(1,2)--(0,2)--(0,1);	
\filldraw[fill=blue,fill opacity=0.8]  (1,1)--(2,1)--(2,2)--(1,2)--(1,1);		
\filldraw[fill=black,fill opacity=0.8]  (0,2)--(0.5,2)--(1,3)--(0,3)--(0,2);
\filldraw[fill=green,fill opacity=0.8]  (0.5,2)--(2,2)--(2,3)--(1,3)--(0.5,2);
\filldraw[fill=red,fill opacity=0.8]  (0,3)--(0,4)--(0.5,4)--(0.5,3)--(0,3);
\filldraw[fill=blue,fill opacity=0.8]  (0.5,3)--(0.5,4)--(1.5,4)--(1.5,3)--(0.5,3);
\filldraw[fill=purple,fill opacity=0.8]  (1.5,3)--(1.5,4)--(2,4)--(2,3)--(1.5,3);
\filldraw[fill=black,fill opacity=0.8]  (0,4)--(0,5)--(1,5)--(0.5,4)--(0,4);
\filldraw[fill=green,fill opacity=0.8]  (0.5,4)--(2,4)--(2,5)--(1,5)--(0.5,4);
\filldraw[fill=blue,fill opacity=0.8]  (0,5)--(0,6)--(1,6)--(1,5)--(0,5);
\filldraw[fill=blue,fill opacity=0.8]  (1,5)--(1,6)--(2,6)--(2,5)--(1,5);
\filldraw[fill=green,fill opacity=0.8]  (0,6)--(0,7)--(1.5,7)--(1,6)--(0,6);
\filldraw[fill=black,fill opacity=0.8]  (1,6)--(2,6)--(2,7)--(1.5,7)--(1,6);
\filldraw[fill=purple,fill opacity=0.8]  (0,7)--(0,8)--(0.5,8)--(0.5,7)--(0,7);
\filldraw[fill=black,fill opacity=0.8]  (0.5,7)--(0.5,8)--(1.5,8)--(1.5,7)--(0.5,7);
\filldraw[fill=red,fill opacity=0.8]  (1.5,7)--(1.5,8)--(2,8)--(2,7)--(1.5,7);

\end{scope}	
%%%%%%%%%%%F_4^{(1)}%%%%%%%%%%%%%%%	
\begin{scope}[shift={(8,0)}]
\draw (0,0)--(0,10)--(2,10)--(2,0);
\draw (0,1)--(2,1);
\draw (0,2)--(2,2);
\draw (0,3)--(2,3);
\draw (0,5)--(2,5);
\draw (0,6)--(2,6);
\draw (0,7)--(2,7);
\draw (0,8)--(2,8);
\draw (0,0)--(1,1);
\draw (1,0)--(1,1);
\draw (1,0)--(1.5,1);
\draw (1,1)--(1,2);
\draw (1,2)--(1,3);
\draw (0.75,3)--(0.75,4);
\draw (0,4)--(1,4)--(1,4.5)--(1.25,4.5)--(1.25,5);
\draw (0.5,4)--(0.5,5);
\draw (1.5,3)--(1.5,5);
\draw (0.5,5)--(1,6);
\draw (1,5)--(1,6);
\draw (1,5)--(2,6);
\draw (1,6)--(1,8);
\draw (0.5,8)--(0.5,10);
\draw (1.25,8)--(1.25,9)--(1,9)--(1,9.5)--(0.75,9.5)--(0.75,10);
\draw (1.25,9)--(2,9);
\draw (1.5,9)--(1.5,10);
\draw (0,0)--(0,-0.5);	
\draw (2,0)--(2,-0.5);
\node at (1,-0.4) {$\vdots$};

\filldraw[fill=green,fill opacity=0.8]  (0,0)--(0,1)--(1.5,1)--(1,0)--(1,1)--(0,0);
\filldraw[fill=purple,fill opacity=0.8]  (0,0)--(1,0)--(1,1)--(0,0);
\filldraw[fill=black,fill opacity=0.8]  (1,0)--(1.5,1)--(2,1)--(2,0)--(1,0);
\filldraw[fill=green,fill opacity=0.8]  (0,1)--(0,2)--(1,2)--(1,1)--(0,1);
\filldraw[fill=blue,fill opacity=0.8]  (1,1)--(1,2)--(2,2)--(2,1)--(1,1);
\filldraw[fill=blue,fill opacity=0.8]  (0,2)--(0,3)--(1,3)--(1,2)--(0,2);
\filldraw[fill=green,fill opacity=0.8]  (1,2)--(1,3)--(2,3)--(2,2)--(1,2);
\filldraw[fill=black,fill opacity=0.8]  (0,3)--(0,4)--(0.75,4)--(0.75,3)--(0,3);
\filldraw[fill=green,fill opacity=0.8]  (0.75,3)--(0.75,4)--(1,4)--(1,4.5)--(1.25,4.5)--(1.25,5)--(1.5,5)--(1.5,3)--(0.75,3);
\filldraw[fill=purple,fill opacity=0.8]  (1.5,3)--(1.5,5)--(2,5)--(2,3)--(1.5,3);
\filldraw[fill=red,fill opacity=0.8]  (0,4)--(0,5)--(0.5,5)--(0.5,4)--(0,4);
\filldraw[fill=blue,fill opacity=0.8]  (0.5,4)--(0.5,5)--(1.25,5)--(1.25,4.5)--(1,4.5)--(1,4)--(0.5,4);
\filldraw[fill=black,fill opacity=0.8]  (0,5)--(0.5,5)--(1,6)--(0,6)--(0,5);
\filldraw[fill=green,fill opacity=0.8]  (1,6)--(0.5,5)--(2,5)--(2,6)--(1,5)--(1,6);
\filldraw[fill=purple,fill opacity=0.8]  (1,5)--(1,6)--(2,6)--(1,5);
\filldraw[fill=blue,fill opacity=0.8]  (0,6)--(0,7)--(1,7)--(1,6)--(0,6);
\filldraw[fill=green,fill opacity=0.8]  (1,6)--(1,7)--(2,7)--(2,6)--(1,6);
\filldraw[fill=green,fill opacity=0.8]  (0,7)--(0,8)--(1,8)--(1,7)--(0,7);
\filldraw[fill=blue,fill opacity=0.8]  (1,7)--(1,8)--(2,8)--(2,7)--(1,7);
\filldraw[fill=purple,fill opacity=0.8]  (0,8)--(0,10)--(0.5,10)--(0.5,8)--(0,8);
\filldraw[fill=green,fill opacity=0.8]  (0.5,8)--(0.5,10)--(0.75,10)--(0.75,9.5)--(1,9.5)--(1,9)--(1.25,9)--(1.25,8)--(0.5,8);
\filldraw[fill=blue,fill opacity=0.8]  (0.75,10)--(0.75,9.5)--(1,9.5)--(1,9)--(1.5,9)--(1.5,10)--(0.75,10);
\filldraw[fill=black,fill opacity=0.8]  (1.25,8)--(2,8)--(2,9)--(1.25,9)--(1.25,8);
\filldraw[fill=red,fill opacity=0.8]  (1.5,9)--(2,9)--(2,10)--(1.5,10)--(1.5,9);
\node at (1,-1.7) {Type $F_4^{(1)}$};
\begin{scope}[shift={(0,10)}]
	\draw (0,10)--(0,10.5);
\draw (2,10)--(2,10.5);
\node at (1,10.6) {$\vdots$};	
\draw (0,0)--(0,10)--(2,10)--(2,0);
\draw (0,1)--(2,1);
\draw (0,2)--(2,2);
\draw (0,3)--(2,3);
\draw (0,5)--(2,5);
\draw (0,6)--(2,6);
\draw (0,7)--(2,7);
\draw (0,8)--(2,8);
\draw (0,0)--(1,1);
\draw (1,0)--(1,1);
\draw (1,0)--(1.5,1);
\draw (1,1)--(1,2);
\draw (1,2)--(1,3);
\draw (0.75,3)--(0.75,4);
\draw (0,4)--(1,4)--(1,4.5)--(1.25,4.5)--(1.25,5);
\draw (0.5,4)--(0.5,5);
\draw (1.5,3)--(1.5,5);
\draw (0.5,5)--(1,6);
\draw (1,5)--(1,6);
\draw (1,5)--(2,6);
\draw (1,6)--(1,8);
\draw (0.5,8)--(0.5,10);
\draw (1.25,8)--(1.25,9)--(1,9)--(1,9.5)--(0.75,9.5)--(0.75,10);
\draw (1.25,9)--(2,9);
\draw (1.5,9)--(1.5,10);
\filldraw[fill=green,fill opacity=0.8]  (0,0)--(0,1)--(1.5,1)--(1,0)--(1,1)--(0,0);
\filldraw[fill=purple,fill opacity=0.8]  (0,0)--(1,0)--(1,1)--(0,0);
\filldraw[fill=black,fill opacity=0.8]  (1,0)--(1.5,1)--(2,1)--(2,0)--(1,0);
\filldraw[fill=green,fill opacity=0.8]  (0,1)--(0,2)--(1,2)--(1,1)--(0,1);
\filldraw[fill=blue,fill opacity=0.8]  (1,1)--(1,2)--(2,2)--(2,1)--(1,1);
\filldraw[fill=blue,fill opacity=0.8]  (0,2)--(0,3)--(1,3)--(1,2)--(0,2);
\filldraw[fill=green,fill opacity=0.8]  (1,2)--(1,3)--(2,3)--(2,2)--(1,2);
\filldraw[fill=black,fill opacity=0.8]  (0,3)--(0,4)--(0.75,4)--(0.75,3)--(0,3);
\filldraw[fill=green,fill opacity=0.8]  (0.75,3)--(0.75,4)--(1,4)--(1,4.5)--(1.25,4.5)--(1.25,5)--(1.5,5)--(1.5,3)--(0.75,3);
\filldraw[fill=purple,fill opacity=0.8]  (1.5,3)--(1.5,5)--(2,5)--(2,3)--(1.5,3);
\filldraw[fill=red,fill opacity=0.8]  (0,4)--(0,5)--(0.5,5)--(0.5,4)--(0,4);
\filldraw[fill=blue,fill opacity=0.8]  (0.5,4)--(0.5,5)--(1.25,5)--(1.25,4.5)--(1,4.5)--(1,4)--(0.5,4);
\filldraw[fill=black,fill opacity=0.8]  (0,5)--(0.5,5)--(1,6)--(0,6)--(0,5);
\filldraw[fill=green,fill opacity=0.8]  (1,6)--(0.5,5)--(2,5)--(2,6)--(1,5)--(1,6);
\filldraw[fill=purple,fill opacity=0.8]  (1,5)--(1,6)--(2,6)--(1,5);
\filldraw[fill=blue,fill opacity=0.8]  (0,6)--(0,7)--(1,7)--(1,6)--(0,6);
\filldraw[fill=green,fill opacity=0.8]  (1,6)--(1,7)--(2,7)--(2,6)--(1,6);
\filldraw[fill=green,fill opacity=0.8]  (0,7)--(0,8)--(1,8)--(1,7)--(0,7);
\filldraw[fill=blue,fill opacity=0.8]  (1,7)--(1,8)--(2,8)--(2,7)--(1,7);
\filldraw[fill=purple,fill opacity=0.8]  (0,8)--(0,10)--(0.5,10)--(0.5,8)--(0,8);
\filldraw[fill=green,fill opacity=0.8]  (0.5,8)--(0.5,10)--(0.75,10)--(0.75,9.5)--(1,9.5)--(1,9)--(1.25,9)--(1.25,8)--(0.5,8);
\filldraw[fill=blue,fill opacity=0.8]  (0.75,10)--(0.75,9.5)--(1,9.5)--(1,9)--(1.5,9)--(1.5,10)--(0.75,10);
\filldraw[fill=black,fill opacity=0.8]  (1.25,8)--(2,8)--(2,9)--(1.25,9)--(1.25,8);
\filldraw[fill=red,fill opacity=0.8]  (1.5,9)--(2,9)--(2,10)--(1.5,10)--(1.5,9);	
\end{scope}
\end{scope}
\end{tikzpicture}
\vspace{-2pt}
\caption{Young column patterns for types $E_{6}^{(2)}$ and $F_{4}^{(1)}$} \label{Young column patterns}
\end{figure}

\begin{defn}\label{def:pre-Young column}
	A {\it pre-Young column} is a continuous part of the Young column pattern such that
	\begin{enumerate}
		\item[(1)] the height is bounded above,
		\item[(2)] there is no empty space below any block.
	\end{enumerate}
\end{defn}

\begin{defn}\label{def:Young column}\hfill
\begin{enumerate}
	\item 	A block in a pre-Young column is \textit{free} if removing it produces another pre-Young column.
	\item A pre-Young column is \textit{exceptional} if
\begin{enumerate}
\item there is precisely one free $2$-block at the top of the column,
\item the column contains a free $1$-block.
\end{enumerate}
\item If a pre-Young column is not exceptional then it is called a \textit{Young column}.
\end{enumerate}	
	
\end{defn}

\begin{defn}\label{def:equivalent}
    Young columns are \textit{equivalent} if they can be obtained from one another by vertical shift and $180^{\circ}$ rotation around the vertical axis.
\end{defn}

In Figure \ref{Young column E} we list all equivalence classes of Young columns in type $E_6^{(2)}$ using Definitions \ref{def:pre-Young column}, \ref{def:Young column} and \ref{def:equivalent}.
Moreover we label each class with an element of the level $1$ perfect crystal $B$.

\input{Young_columns_E}

\begin{Rmk}
The following is the unique exceptional pre-Young column in type $E_{6}^{(2)}$.
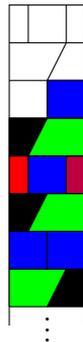
\begin{figure}[H]
	\centering
\begin{tikzpicture}[scale=0.5]
\draw (0,0)--(0,8)--(2,8)--(2,0);
		\draw (0,7)--(2,7);
		\draw (0,6)--(2,6);	
		\draw (0,5)--(2,5);		
		\draw (0,4)--(2,4);		
		\draw (0,3)--(2,3);		
		\draw (0,2)--(2,2);		
		\draw (0,1)--(2,1);		
		\draw (0.5,3)--(0.5,4)--(1,5)--(1,6)--(1.5,7)--(1.5,8);	
		\draw 	(0.5,7)--(0.5,8);
		\draw (0.5,2)--(1,3);
		\draw (1,1)--(1,2);
		\draw (1,0)--(1.5,1);
		\draw (1.5,3)--(1.5,4);	
		\draw (0,0)--(0,-0.5);	
		\draw (2,0)--(2,-0.5);
		\node at (1,-0.4) {$\vdots$};
		
		\filldraw[fill=green,fill opacity=0.8]  (0,0)--(1,0)--(1.5,1)--(0,1)--(0,0);		
		\filldraw[fill=black,fill opacity=0.8]  (1,0)--(2,0)--(2,1)--(1.5,1)--(1,0);	
		\filldraw[fill=blue,fill opacity=0.8]  (0,1)--(1,1)--(1,2)--(0,2)--(0,1);	
		\filldraw[fill=blue,fill opacity=0.8]  (1,1)--(2,1)--(2,2)--(1,2)--(1,1);		
		\filldraw[fill=black,fill opacity=0.8]  (0,2)--(0.5,2)--(1,3)--(0,3)--(0,2);
		\filldraw[fill=green,fill opacity=0.8]  (0.5,2)--(2,2)--(2,3)--(1,3)--(0.5,2);
		\filldraw[fill=red,fill opacity=0.8]  (0,3)--(0,4)--(0.5,4)--(0.5,3)--(0,3);
		\filldraw[fill=blue,fill opacity=0.8]  (0.5,3)--(0.5,4)--(1.5,4)--(1.5,3)--(0.5,3);
		\filldraw[fill=purple,fill opacity=0.8]  (1.5,3)--(1.5,4)--(2,4)--(2,3)--(1.5,3);
		\filldraw[fill=black,fill opacity=0.8]  (0,4)--(0,5)--(1,5)--(0.5,4)--(0,4);
		\filldraw[fill=green,fill opacity=0.8]  (0.5,4)--(2,4)--(2,5)--(1,5)--(0.5,4);
		%\filldraw[fill=blue,fill opacity=0.8]  (0,5)--(0,6)--(1,6)--(1,5)--(0,5);
		\filldraw[fill=blue,fill opacity=0.8]  (1,5)--(1,6)--(2,6)--(2,5)--(1,5);
		%\filldraw[fill=green,fill opacity=0.8]  (0,6)--(0,7)--(1.5,7)--(1,6)--(0,6);
		%\filldraw[fill=black,fill opacity=0.8]  (1,6)--(2,6)--(2,7)--(1.5,7)--(1,6);
		%\filldraw[fill=purple,fill opacity=0.8]  (0,7)--(0,8)--(0.5,8)--(0.5,7)--(0,7);
		%\filldraw[fill=black,fill opacity=0.8]  (0.5,7)--(0.5,8)--(1.5,8)--(1.5,7)--(0.5,7);
		%\filldraw[fill=red,fill opacity=0.8]  (1.5,7)--(1.5,8)--(2,8)--(2,7)--(1.5,7);

	\end{tikzpicture}	
\end{figure} 	
\end{Rmk}

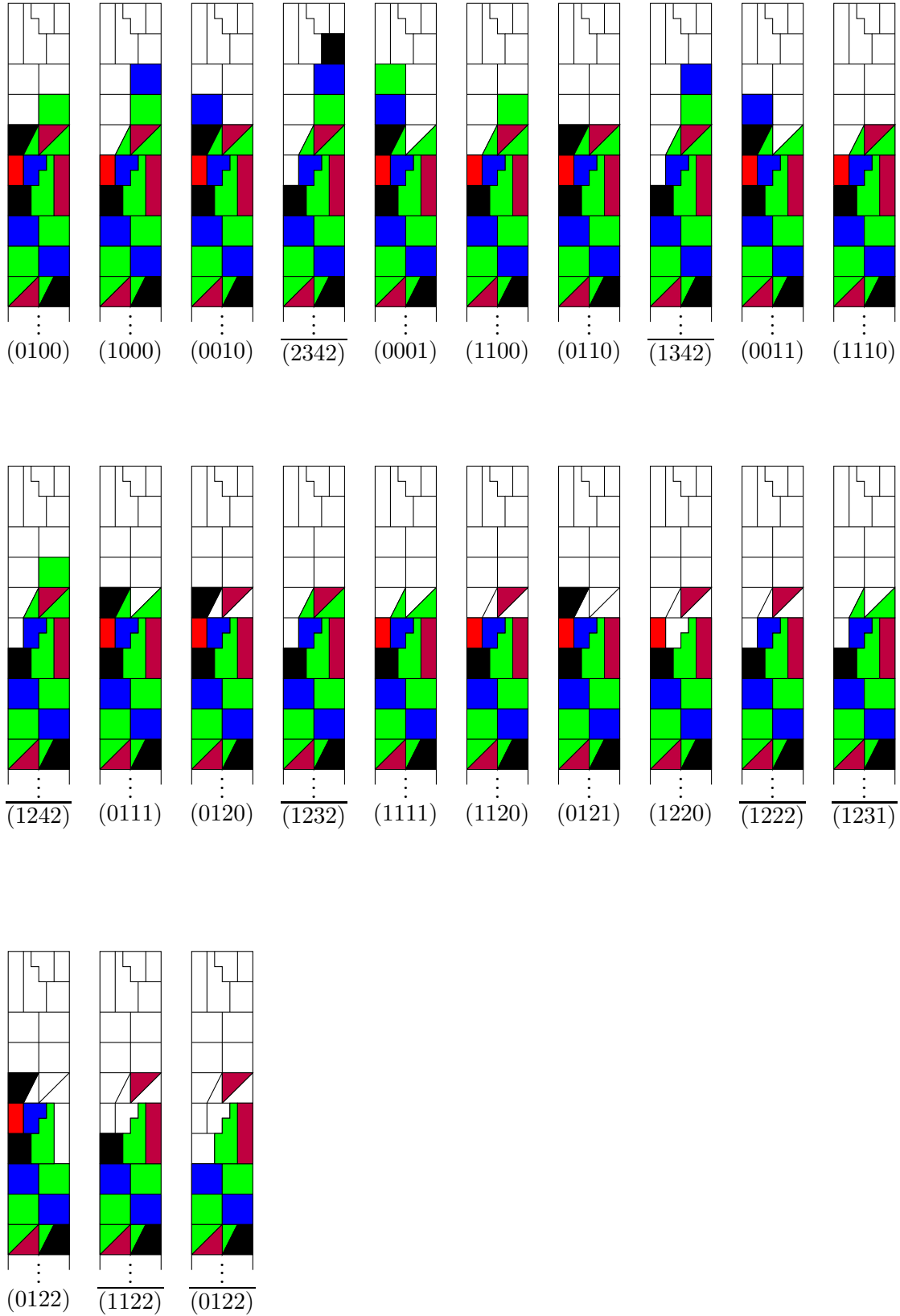
The equivalence classes of Young columns in type $F_4^{(1)}$ are given in Figure \ref{Young column F}.

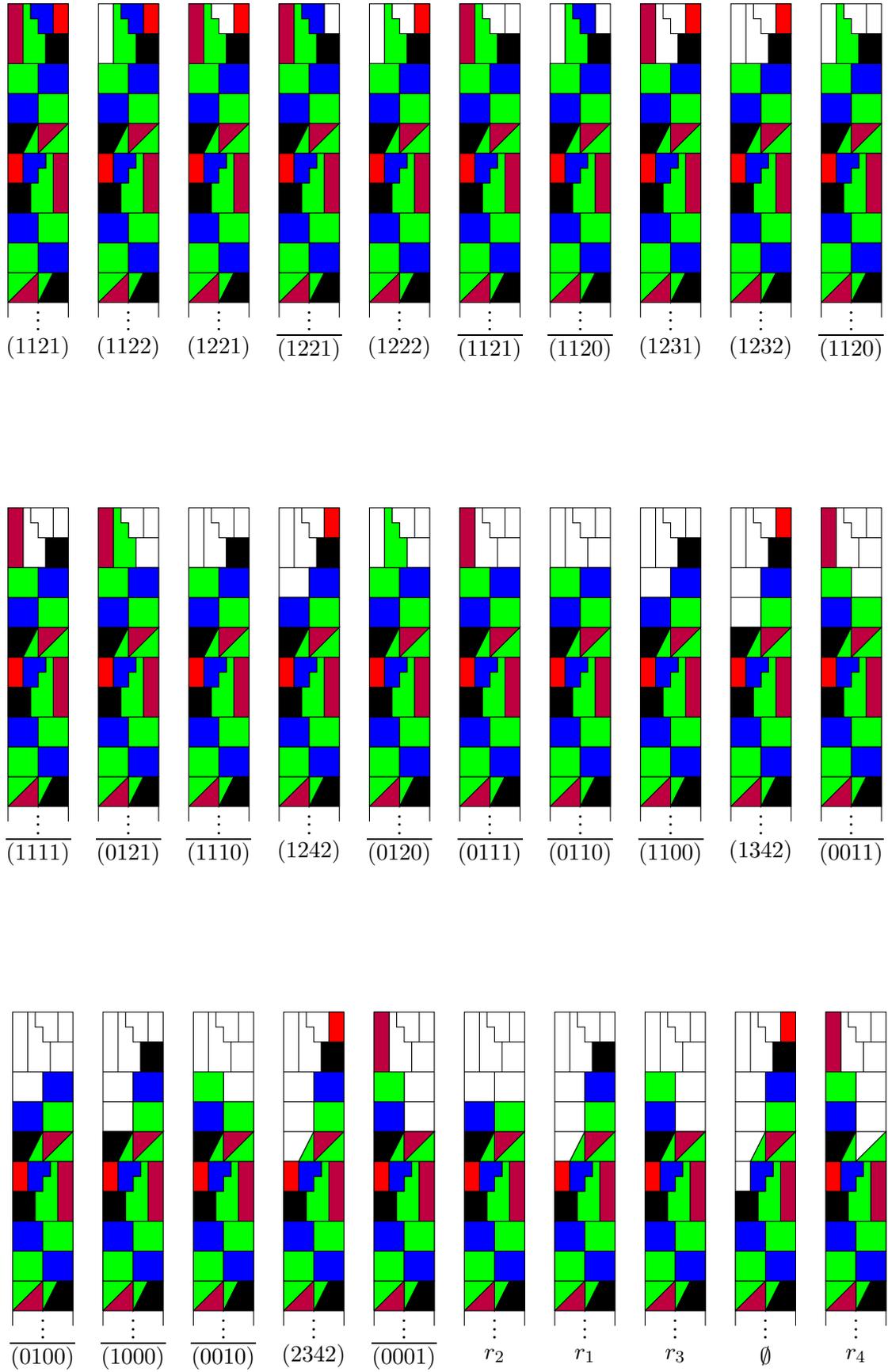
\input{Young_columns_F}

\begin{Rmk}
    The following is the unique exceptional pre-Young column in type $F_4^{(1)}$.
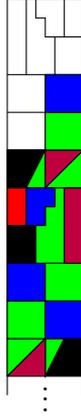
\begin{figure}[H]
\centering
\begin{tikzpicture}[scale=0.5]
\draw (0,0)--(0,10)--(2,10)--(2,0);
\draw (0,1)--(2,1);
\draw (0,2)--(2,2);
\draw (0,3)--(2,3);
\draw (0,5)--(2,5);
\draw (0,6)--(2,6);
\draw (0,7)--(2,7);
\draw (0,8)--(2,8);
\draw (0,0)--(1,1);
\draw (1,0)--(1,1);
\draw (1,0)--(1.5,1);
\draw (1,1)--(1,2);
\draw (1,2)--(1,3);
\draw (0.75,3)--(0.75,4);
\draw (0,4)--(1,4)--(1,4.5)--(1.25,4.5)--(1.25,5);
\draw (0.5,4)--(0.5,5);
\draw (1.5,3)--(1.5,5);
\draw (0.5,5)--(1,6);
\draw (1,5)--(1,6);
\draw (1,5)--(2,6);
\draw (1,6)--(1,8);
\draw (0.5,8)--(0.5,10);
\draw (1.25,8)--(1.25,9)--(1,9)--(1,9.5)--(0.75,9.5)--(0.75,10);
\draw (1.25,9)--(2,9);
\draw (1.5,9)--(1.5,10);
\draw (0,0)--(0,-0.5);	
\draw (2,0)--(2,-0.5);
\node at (1,-0.4) {$\vdots$};

\filldraw[fill=green,fill opacity=0.8]  (0,0)--(0,1)--(1.5,1)--(1,0)--(1,1)--(0,0);
\filldraw[fill=purple,fill opacity=0.8]  (0,0)--(1,0)--(1,1)--(0,0);
\filldraw[fill=black,fill opacity=0.8]  (1,0)--(1.5,1)--(2,1)--(2,0)--(1,0);
\filldraw[fill=green,fill opacity=0.8]  (0,1)--(0,2)--(1,2)--(1,1)--(0,1);
\filldraw[fill=blue,fill opacity=0.8]  (1,1)--(1,2)--(2,2)--(2,1)--(1,1);
\filldraw[fill=blue,fill opacity=0.8]  (0,2)--(0,3)--(1,3)--(1,2)--(0,2);
\filldraw[fill=green,fill opacity=0.8]  (1,2)--(1,3)--(2,3)--(2,2)--(1,2);
\filldraw[fill=black,fill opacity=0.8]  (0,3)--(0,4)--(0.75,4)--(0.75,3)--(0,3);
\filldraw[fill=green,fill opacity=0.8]  (0.75,3)--(0.75,4)--(1,4)--(1,4.5)--(1.25,4.5)--(1.25,5)--(1.5,5)--(1.5,3)--(0.75,3);
\filldraw[fill=purple,fill opacity=0.8]  (1.5,3)--(1.5,5)--(2,5)--(2,3)--(1.5,3);
\filldraw[fill=red,fill opacity=0.8]  (0,4)--(0,5)--(0.5,5)--(0.5,4)--(0,4);
\filldraw[fill=blue,fill opacity=0.8]  (0.5,4)--(0.5,5)--(1.25,5)--(1.25,4.5)--(1,4.5)--(1,4)--(0.5,4);
\filldraw[fill=black,fill opacity=0.8]  (0,5)--(0.5,5)--(1,6)--(0,6)--(0,5);
\filldraw[fill=green,fill opacity=0.8]  (1,6)--(0.5,5)--(2,5)--(2,6)--(1,5)--(1,6);
\filldraw[fill=purple,fill opacity=0.8]  (1,5)--(1,6)--(2,6)--(1,5);
%	\filldraw[fill=blue,fill opacity=0.8]  (0,6)--(0,7)--(1,7)--(1,6)--(0,6);
\filldraw[fill=green,fill opacity=0.8]  (1,6)--(1,7)--(2,7)--(2,6)--(1,6);
%	\filldraw[fill=green,fill opacity=0.8]  (0,7)--(0,8)--(1,8)--(1,7)--(0,7);
\filldraw[fill=blue,fill opacity=0.8]  (1,7)--(1,8)--(2,8)--(2,7)--(1,7);
%\filldraw[fill=purple,fill opacity=0.8]  (0,8)--(0,10)--(0.5,10)--(0.5,8)--(0,8);
%\filldraw[fill=green,fill opacity=0.8]  (0.5,8)--(0.5,10)--(0.75,10)--(0.75,9.5)--(1,9.5)--(1,9)--(1.25,9)--(1.25,8)--(0.5,8);
%\filldraw[fill=blue,fill opacity=0.8]  (0.75,10)--(0.75,9.5)--(1,9.5)--(1,9)--(1.5,9)--(1.5,10)--(0.75,10);
%\filldraw[fill=black,fill opacity=0.8]  (1.25,8)--(2,8)--(2,9)--(1.25,9)--(1.25,8);
%\filldraw[fill=red,fill opacity=0.8]  (1.5,9)--(2,9)--(2,10)--(1.5,10)--(1.5,9);	
\end{tikzpicture}	
\end{figure} 	
\end{Rmk}

In each type we denote by $y_v$ the equivalence class of Young columns labelled by an element $v$ of the level $1$ perfect crystal $B$ or $B'$, as displayed in Figures \ref{Young column E} and \ref{Young column F}.

\begin{defn}
Let $y$ be a Young column.
    \begin{enumerate}
        \item An $i$-block in $y$ is \textit{removable} if removing it from $y$ produces another Young column.
        \item An $i$-block in the Young column pattern which is not in $y$ is \textit{addable} if adding it to $y$ produces another Young column.
    \end{enumerate}
\end{defn}

We can endow the set of Young columns with the structure of an affine crystal.
\begin{defn} \label{def:Young column crystal structure}
    \begin{enumerate}
        \item $\ft_{i}$ acts on a Young column $y$ by adding an addable $i$-block if it exists and mapping to $0$ otherwise, with the following caveats.
        \begin{enumerate}
            \item If $y$ has two addable $i$-blocks then applying $\ft_{i}$ adds the higher one.
            \item $\ft_{i}(y) = 0$ whenever $y$ lies in an equivalence class $y_{r_{j}}$ for some $j\not= i$.
        \end{enumerate}
        \item $\et_{i}$ acts on a Young column $y$ by removing a removable $i$-block if it exists and mapping to $0$ otherwise, with the following caveats.
        \begin{enumerate}
            \item If $y$ has two removable $i$-blocks then applying $\et_{i}$ removes the lower one.
            \item $\et_{i}(y) = 0$ whenever $y$ lies in an equivalence class $y_{r_{j}}$ for some $j\not= i$.
        \end{enumerate}
        \item Define $\varphi_{i}(y) = \max\lbrace n ~\vert~ \ft_{i}^{n}y\neq 0\rbrace$ and $\varepsilon_{i}(y) = \max\lbrace n ~\vert~ \et_{i}^{n}y\neq 0\rbrace$.
        \item Fix $\wt(y) = \Lambda_{0}$ for some $y$ in the equivalence class $y_{\emptyset}$ and extend to all Young columns with the conditions $\wt(\ft_{i}y) = \wt(y) - \alpha_{i}$ if $\ft_{i}y\not= 0$ and $\wt(\et_{i}y) = \wt(y) + \alpha_{i}$ if $\et_{i}y\not= 0$ from Definition \ref{def:crystal}.
    \end{enumerate}
\end{defn}

By projecting the weights to $\Pbar$ this descends to a classical crystal structure on the set of \textit{equivalence classes} of Young columns, which we denote by $C$ and $C'$ in types $E_{6}^{(2)}$ and $F_{4}^{(1)}$ respectively.
Conversely, the affinizations of $C$ and $C'$ are precisely the original affine crystals of Young columns.

Comparing the crystals graphs in Figures \ref{crystal graph of B} and \ref{perfect graph B'} with these classical crystal structures and our Young column patterns, we see that $C$ and $C'$ provide us with combinatorial models for $B$ and $B'$.

\begin{Prop} \label{Young column realization}
    The map $\phi : y_{v} \mapsto v$ defines an isomorphism of classical crystals $C \xrightarrow{\sim} B$ (resp. $C' \xrightarrow{\sim} B'$).
\end{Prop}

\begin{Rmk}
    The caveats in Definition \ref{def:Young column crystal structure} are to ensure that we have well-defined $\et_{i}$ and $\ft_{i}$ maps, $i$-strings
    $y_{\overline{(abcd)}}\xrightarrow{0} y_{\emptyset} \xrightarrow{0} y_{(abcd)}$
    and
    $y_{(abcd)}\xrightarrow{j} y_{r_{j}} \xrightarrow{j} y_{\overline{(abcd)}}$
    for each $j\not= 0$, and no other arrows incident to any $y_{r_{j}}$.
\end{Rmk}

The crystal graphs of $C$ and $C'$ are displayed in Appendix \ref{realization of B and B'}.
We emphasise that their construction has been purely combinatorial, and in particular independent of the algebraic theory originally used by Benkart-Frenkel-Kang-Lee to define the crystals $B$ and $B'$.

\section{Young wall models for the level \texorpdfstring{$1$}{1} highest weight crystals of \texorpdfstring{$U_q(E_6^{(2)})$}{Uq(E6(2))} and \texorpdfstring{$U_q(F_4^{(1)})$}{Uq(F4(1))}} \label{Sec:highest weight Young wall realization}

We now combine our Young column models for $B$ and $B'$ with the path realization of Section \ref{Sec:Crystals} in order to obtain Young wall models for the level $1$ irreducible highest weight crystals $B(\lambda)$ in types $E_6^{(2)}$ and $F_4^{(1)}$.

Recall that in each type the unique weight $\lambda\in\Pbar^{+}$ of level $1$ is $\lambda = \Lambda_{0}$, with minimal vectors $b_{\lambda} = b^{\lambda} = \emptyset$ in $B$ and $B'$ and ground-state path $\pb_{\lambda} = (\emptyset)_{k=0}^{\infty}$.

Arranging Young columns from the corresponding equivalence class $y_{\emptyset}$ at the same height and orientation produces the \emph{Young wall patterns} and \emph{ground-state walls}.

\input{Young_wall_patterns}

\input{Ground_state_walls}

We call the columns of the ground-state wall the {\it ground-state columns}.

\begin{defn} \label{def:Young wall}
    In each type, a Young wall is a collection of blocks stacked inside the Young wall pattern such that
    \begin{enumerate}
        \item it differs from the ground-state wall in finitely many blocks,
        \item each column of the wall is a Young column.
    \end{enumerate}
\end{defn}

Many papers assume two further conditions for their Young walls, the first of which we shall call the \emph{right block property}:
\begin{itemize}
    \item[\textendash] if a Young wall contains a block, then it must contain the block occupying the same position in the column to the right,
    $\hfill \refstepcounter{equation}(\theequation)\label{right block property}$
    \item[\textendash] a Young wall must be built on top of the ground-state wall.
    $\hfill \refstepcounter{equation}(\theequation)\label{built on ground-state wall property}$
\end{itemize}
We have removed these assumptions from our definition since it is not immediately clear that they should hold for the Young walls in our models for $B(\lambda)$.
Indeed, we shall see in Section \ref{Sec:Fock space Young wall realization} that they do not hold in general for the walls in our models for the Fock space crystals.

Nevertheless, with Proposition \ref{highest weight right block property proposition} and Corollary \ref{highest weight built on ground-state wall corollary} respectively, we prove that these conditions are in fact satisfied by the \textit{reduced} Young walls which form our models for $B(\lambda)$.

Throughout this section we shall usually write a Young wall $Y$ as a sequence $(\dots,y_2,y_1,y_0)$ of Young columns, considered only up to equivalence as elements of $C$ or $C'$.
Let us denote by $|y_k|$ (resp. ${|y_k|}_0$) the difference in the number of blocks (resp. $0$-blocks) between $Y$ and the ground-state wall in column $k$.

Recall that Appendix \ref{appe:HandH'values} contains the values of the energy functions $H$ on $B\cong C$ and $H'$ on $B'\cong C'$.

\begin{defn} \label{def:reduced} \hfill
\begin{enumerate}
    \item A pair of adjacent columns $(y_{k+1},y_{k})$ in a Young wall $Y$ in type $E_{6}^{(2)}$ is reduced if
    \begin{equation} \label{reduced equation}
    H(y_{k+1}\otimes y_{k}) + |y_{k+1}|_{0} - |y_{k}|_{0}
    = H^{\mathrm{aff}}(y_{k+1}({|y_{k+1}|}_0)\otimes y_k({|y_k|}_0))
    = 0,
    \end{equation}
    and similarly in type $F_{4}^{(1)}$ with $H$ replaced by $H'$.
    \item A Young wall $Y$ is reduced if every pair $(y_{k+1},y_{k})$ is reduced.
\end{enumerate}
\end{defn}

We shall denote the set of reduced Young walls by $\Ycal(\lambda)$.

\begin{Prop} \label{highest weight right block property proposition}
    In types $E_{6}^{(2)}$ and $F_{4}^{(1)}$ every reduced Young wall satisfies the right block property.
\end{Prop}
\begin{proof}
 This is an immediate consequence of Proposition \ref{Fock right block proposition}.
\end{proof}

\begin{Cor} \label{highest weight built on ground-state wall corollary}
    In types $E_{6}^{(2)}$ and $F_{4}^{(1)}$ every reduced Young wall is built on top of the ground-state wall.
\end{Cor}
\begin{proof}
    This follows from the right block property above, since a Young wall differs from the ground-state wall in finitely many blocks and thus matches it in all columns sufficiently far to the left.
\end{proof}

\begin{Prop}\label{adjacent column}
    If a pair of adjacent columns $(y_{k+1},y_{k})$ in a Young wall is reduced then $|y_{k}| - |y_{k+1}|$ is a fixed non-negative integer.
\end{Prop}

\begin{proof}
    We proceed as in \cite{FHKS23}*{Proposition 3.14} and \cite{Lau23}*{Proposition 4.6}.
    Namely, for each choice of $y_{k+1}$ and $y_{k}$ (up to equivalence) there is precisely one value of $|y_{k}|_{0} - |y_{k+1}|_{0}$ such that $(y_{k+1},y_{k})$ is reduced by definition.
    By inspecting the Young column patterns of Figure \ref{Young column patterns} we see that this in turn fixes $|y_{k}| - |y_{k+1}|$, which must be non-negative by Proposition \ref{highest weight right block property proposition}.
\end{proof}

It follows that up to vertical shift there are exactly $|B|^{2} = 729$ and $|B'|^{2} = 2809$ pairs of reduced adjacent columns in types $E_{6}^{(2)}$ and $F_{4}^{(1)}$ respectively.

We shall now define the structure of an affine crystal on the set of reduced Young walls $\Ycal(\lambda)$.
Recall from Definition \ref{def:Young column crystal structure} (3) that $\varphi_{i}(y)$ (resp. $\varepsilon_{i}(y)$) is the maximum number of $i$-blocks which can be added to (resp. removed from) a Young column $y$ sequentially, while still remaining a Young column.

\begin{defn}
    The \textit{$i$-signature} of $y$ is the sequence $\sign_{i}(y) = \underbrace{-\dots-}_{\varepsilon_{i}(y)}\underbrace{+\dots+}_{\varphi_{i}(y)}$.
\end{defn}

For each Young wall $Y = (\dots,y_2,y_1,y_0)$ we define the \textit{pre-$i$-signature} to be the (possibly infinite) sequence
\begin{equation*}
    \presign_{i}(Y) = \dots\sign_{i}(y_{2})\sign_{i}(y_{1})\sign_{i}(y_{0})
\end{equation*}
of $+$'s and $-$'s.
Cancelling every $+-$ pair leaves a finite number of $-$'s followed by a finite number of $+$'s, reading from left to right, called the \textit{$i$-signature} $\sign_{i}(Y)$ of $Y$.

\medskip

We define $\tilde{E}_i Y$ to be the Young wall obtained from $Y$ by applying $\et_{i}$ to the column containing the rightmost $-$ in $\sign_{i}(Y)$ if it exists, and $0$ otherwise.

\medskip

We define $\tilde{F}_i Y$ to be the Young wall obtained from $Y$ by applying $\ft_{i}$ to the column containing the leftmost $+$ in $\sign_{i}(Y)$ if it exists, and $0$ otherwise.

\medskip

This is called the \textit{tensor product rule} for Young walls.

\begin{Prop}\label{prop:close}
    For any $Y\in\mathcal Y(\lambda)$ we have $\tilde{E}_iY,\tilde{F}_iY\in\mathcal Y(\lambda)\cup \{0\}$.
\end{Prop} 

\begin{proof}
    Let $Y=(\dots,y_{k+1},y_k,y_{k-1},\dots,y_2,y_1,y_0)$ be a reduced Young wall.
    If $\tilde{F}_iY = 0$ then we are done, so instead suppose that
    $$\tilde{F}_iY=(\dots,y_{k+1},z_k,y_{k-1},\dots,y_2,y_1,y_0)$$
    where $z_k$ is obtained by adding an addable $i$-block to $y_k$.
    
    From the tensor product rule for Young walls it is easy to see that $\varphi_{i}(y_{k+1})<\varepsilon_{i}(y_k)$ and $\varphi_{i}(y_{k})>\varepsilon_{i}(y_{k-1})$, hence by (\ref{tensor product of crystals}) we have $\ft_{i}(y_{k+1},y_k) = (y_{k+1},z_k)$ and $\ft_{i}(y_{k},y_{k-1}) = (z_k,y_{k-1})$.
    It then follows from Theorem \ref{constant} that
	\begin{align*}
		&H^{\mathrm{aff}}(y_{k+1}({|y_{k+1}|}_0)\otimes y_k({|y_k|}_0
		))=H^{\mathrm{aff}}(y_{k+1}({|y_{k+1}|}_0)\otimes z_k({|z_k|}_0
		))=0,\\
		&H^{\mathrm{aff}}(y_{k}({|y_{k}|}_0)\otimes y_{k-1}({|y_{k-1}|}_0
		))=H^{\mathrm{aff}}(z_{k}({|z_{k}|}_0)\otimes  y_{k-1}({| y_{k-1}|}_0
		))=0.
	\end{align*}
    
    Since all other pairs of adjacent columns in $\tilde{F}_iY$ are the same as in $Y$ and thus satisfy (\ref{reduced equation}), the Young wall $\tilde{F}_iY$ is reduced.
    One can prove that $\tilde{E}_iY\in\mathcal Y(\lambda)\cup \{0\}$ in a similar manner.
\end{proof}

Let us furthermore define maps
$\varepsilon_i,\varphi_i : \mathcal Y(\lambda) \longrightarrow \Z$
and
$\wt : \mathcal Y(\lambda) \longrightarrow P$
by
\begin{equation*}
    \begin{aligned}\mbox{}
        \varepsilon_i(Y) &= \text{the number of $-$'s in } \sign_{i}(Y),\\
        \varphi_i(Y) &= \text{the number of $+$'s in } \sign_{i}(Y),\\
        \wt(Y) &= \lambda - \sum_{i\in I} k_i \alpha_i,
	\end{aligned}
\end{equation*}
where $k_i$ is the number of $i$-blocks in $Y$ that have been added to the ground-state wall.
The following result is then proved via a routine check.

\begin{Thm}\label{thm:affine crystal structure}
    The maps
    $\tilde{E}_i,\tilde{F}_i : \mathcal Y(\lambda) \rightarrow \mathcal Y(\lambda)\cup \{0\}$,
    $\varepsilon_i,\varphi_i : \mathcal Y(\lambda) \rightarrow \Z$
    and $\wt : \mathcal Y(\lambda) \rightarrow P$ defined above endow the set of reduced Young walls $\mathcal Y(\lambda)$ with the structure of an affine crystal.
\end{Thm}

\begin{Thm}\label{thm:main}
    In types $E_{6}^{(2)}$ and $F_{4}^{(1)}$ there exists an isomorphism of affine crystals
	\begin{equation*}
		\mathcal Y(\lambda) \stackrel{\sim} \longrightarrow B(\lambda)
		\ \ \text{given by} \ \ 
		Y_{\lambda} \longmapsto u_{\lambda}
	\end{equation*}
    where $u_{\lambda}$ is the highest weight vector in $B(\lambda)$.
\end{Thm}
\begin{proof}
    From the path realization of Proposition \ref{prop:path realization} it is enough to show that $\mathcal Y(\lambda)\cong\mathcal P(\lambda)$. With the crystal isomorphism $\phi$ from Proposition \ref{Young column realization} we define a map $\Phi:\mathcal Y(\lambda)\to\mathcal P(\lambda)$ by
	\begin{align*}
		\Phi:(\dots,y_1,y_0)\mapsto (\dots,\phi(y_1),\phi(y_0)).
	\end{align*}
    
    Using the tensor product rule for Young walls, it is straightforward to check that $\Phi$ commutes with the crystal operators, i.e.
    $\tilde{e}_{i} \circ \Phi = \Phi \circ \tilde{E}_{i}$,
    $\tilde{f}_{i} \circ \Phi = \Phi \circ \tilde{F}_{i}$, and so on.

    Suppose that Young walls $Y=(\dots,y_1,y_0)$ and $Z=(\dots,z_1,z_0)$ are mapped by $\Phi$ to the same path in $\mathcal P(\lambda)$.
    Then $y_{k} = z_{k}$ and $y_{k+1} = z_{k+1}$ as elements of $C$ (resp. $C'$) for all $k\geq 0$.
    But since each $|y_k|-|y_{k+1}|=|z_k|-|z_{k+1}|$ by Proposition \ref{adjacent column}, and moreover $|y_k|=|z_k|$ for $k\gg 0$, it follows that $Y=Z$ and hence $\Phi$ is injective.
    
    For a given path $\mathbf{p}=(p_{k})^{\infty}_{k=0}$ in $\mathcal P(\lambda)$ we can first draw reduced adjacent columns $(y_1,y_0)$ corresponding to $(p_{1},p_{0})$, then attach a column $y_2$ corresponding to $p_{2}$ such that $(y_{2},y_{1})$ is reduced, and so on.
    By specifying that these columns eventually match up with the ground-state columns, we obtain a reduced Young wall $(\dots,y_1,y_0)$ which is sent by $\Phi$ to $\mathbf{p}$, hence $\Phi$ is surjective.
	
    It is clear that the resulting map $\mathcal{Y}(\lambda) \rightarrow B(\lambda)$ sends the ground-state wall to $u_{\lambda}$ and so our proof is complete.
\end{proof}

Appendix \ref{top crystal} displays the top part of the crystal $\mathcal Y(\Lambda_0)$ in each type.

\section{Young wall models for the level \texorpdfstring{$1$}{1} Fock space crystals of \texorpdfstring{$U_q(E_6^{(2)})$}{Uq(E6(2))} and \texorpdfstring{$U_q(F_4^{(1)})$}{Uq(F4(1))}}\label{Sec:Fock space Young wall realization}

Once again, in types $E_{6}^{(2)}$ and $F_{4}^{(1)}$ we let $\lambda = \Lambda_{0}$ be the unique level $1$ dominant integral weight in $\Pbar^{+}$, with minimal vectors $b_{\lambda} = b^{\lambda} = \emptyset$ in $B$ and $B'$ respectively and ground-state path $\pb_{\lambda} = (\emptyset)_{k=0}^{\infty}$.

Since $B$ and $B'$ are the crystal bases of good $U'_q(\mathfrak{g})$-modules by Proposition \ref{BFKL crystals are good}, they can each be used to construct the Fock space crystal $B(\Fcal(\lambda))$.

The ground-state sequence for the Fock space is $(\emptyset(k))_{k=0}^{\infty}$ since $H(\emptyset\otimes\emptyset) = 0$, and by arranging the corresponding Young columns at the same height and orientation we recover precisely the Young wall patterns and ground-state walls of Figures \ref{Young wall patterns} and \ref{ground-state walls}.

Throughout the remainder of this paper, Young walls shall be written as sequences $(y_{k}(n_{k}))_{k=0}^{\infty}$ of Young columns \textit{not up to equivalence}.

In particular, to any sequence $(p_{k}(n_{k}))_{k=0}^{\infty}$ in $\Baff$ (resp. $B'^{\aff}$) which stabilises to the ground-state sequence we can assign a unique Young wall $(y_{k}(n_{k}))_{k=0}^{\infty}$ with $\phi(y_{k}) = p_{k}$ for all $k\geq 0$.

Recall from Section \ref{Fock space preliminaries} that $(p_{k}(n_{k}))_{k=0}^{\infty}$ lies inside the Fock space crystal $B(\Fcal(\lambda))$ when it is normally ordered, whereby $(y_{k}(n_{k}))_{k=0}^{\infty}$ satisfies
\begin{align*}
    H^{\aff}(y_{k+1}(n_{k+1})\otimes y_{k}(n_{k})) < 2
    ~~\quad(\mathrm{resp.~} {(H')}^{\aff}(y_{k+1}(n_{k+1})\otimes y_{k}(n_{k})) < 2 )
\end{align*}
for all $k\geq 0$.
Combining this with Definition \ref{affine energy function} and the identity $n_{k} = k + |y_{k}|_{0}$ this condition becomes
\begin{equation} \label{proper equation}
    H(y_{k+1}\otimes y_{k}) + |y_{k+1}|_{0} - |y_{k}|_{0} \leq 0
    ~~\quad(\mathrm{resp.~} H'(y_{k+1}\otimes y_{k}) + |y_{k+1}|_{0} - |y_{k}|_{0} \leq 0 ).
\end{equation}

\begin{defn}
    A Young wall $(y_{k}(n_{k}))_{k=0}^{\infty}$ is \textit{proper} if it satisfies condition (\ref{proper equation}) for all $k\geq 0$.
\end{defn}

Denote the set of proper Young walls by $\Zcal(\lambda)$.
We can endow $\Zcal(\lambda)$ with the structure of an affine crystal exactly as we did for $\Ycal(\lambda)$ in Section \ref{Sec:highest weight Young wall realization}, using the notions of pre-$i$-signatures and $i$-signatures:
\begin{itemize}
    \item[\textendash] $\tilde{E}_{i}$ acts on the column corresponding to the rightmost $-$ in $\sign_{i}(Y)$,
    \item[\textendash] $\tilde{F}_{i}$ acts on the column corresponding to the leftmost $+$ in $\sign_{i}(Y)$,
    \item[\textendash] $\varepsilon_{i}(Y) =$ number of $-$'s in $\sign_{i}(Y)$,
    \item[\textendash] $\varphi_{i}(Y) =$ number of $+$'s in $\sign_{i}(Y)$,
    \item[\textendash] $\wt(Y) = \lambda - \sum_{i\in I} k_{i} \alpha_{i}$,
\end{itemize}
where $k_{i}$ is the difference in the number of $i$-blocks between $Y$ and the ground-state wall.
An almost identical proof to that of Proposition \ref{prop:close} shows that $\tilde{E}_{i},\tilde{F}_{i} : \Zcal(\lambda) \rightarrow \Zcal(\lambda)\cup\lbrace 0\rbrace$, while the following is verified by another routine check.

\begin{Thm}
    The maps $\tilde{E}_{i},\tilde{F}_{i} : \Zcal(\lambda) \rightarrow \Zcal(\lambda)\cup\lbrace 0\rbrace$, $\varepsilon_{i},\varphi_{i} : \Zcal(\lambda) \rightarrow \Zbb\cup\lbrace -\infty\rbrace$ and $\wt : \Zcal(\lambda) \rightarrow P$ defined above endow $\Zcal(\lambda)$ with the structure of an affine crystal.
\end{Thm}

The proper Young walls thus provide a combinatorial model for the level $1$ Fock space crystal.

\begin{Thm}
    In types $E_{6}^{(2)}$ and $F_{4}^{(1)}$ there exists an isomorphism of affine crystals
	\begin{equation*}
		\mathcal Z(\lambda) \stackrel{\sim} \longrightarrow B(\Fcal(\lambda))
		\ \ \text{with} \ \ 
		Y_{\lambda} \longmapsto (b_{k}(m_{k}))_{k=0}^{\infty}
	\end{equation*}
    where $(b_{k}(m_{k}))_{k=0}^{\infty}$ is the ground-state sequence for the Fock space.
\end{Thm}
\begin{proof}
    Consider the map introduced above which assigns a Young wall to each sequence in $\Baff$ (resp. $B'^{\aff}$) that stabilises to the ground-state sequence.
    It is immediate from the construction that this restricts to a bijection between $B(\Fcal(\lambda))$ and $\Zcal(\lambda)$, and moreover respects the affine crystal structures.
\end{proof}

Note that since the trivial embedding of affine crystals $\Ycal(\lambda) \hookrightarrow \Zcal(\lambda)$ surjects onto to the connected component of the ground-state wall, Figures \ref{top part 0 E62} and \ref{top part 0 F41} in Appendix \ref{top crystal} display the top part of the crystal $\Zcal(\lambda)$ as well as that of $\Ycal(\lambda)$.

\subsection{Structure of the proper Young walls} \label{proper structure}
In a similar manner to \cite{Lau23}*{§5.1} we shall investigate the structure of the proper Young walls lying inside $\Zcal(\lambda)$, which form our model for the Fock space crystal $B(\Fcal(\lambda))$.

Consider the following \textit{local right block property} for a pair of adjacent columns in a Young wall $Y$.
\begin{itemize}
    \item[\textendash] If $Y$ contains a block in column $k+1$ then it contains the block occupying the same position in column $k$.
    $\hfill \refstepcounter{equation}(\theequation)\label{adjacent right block property}$
\end{itemize}

A Young wall $Y$ then satisfies the right block property (\ref{right block property}) if condition (\ref{adjacent right block property}) holds for all $k\geq 0$.
The main result of this subsection is that proper Young walls satisfy a certain \textit{slightly weakened version} of the right block property (\ref{right block property}).

\begin{Prop} \label{Fock right block proposition}
    In types $E_{6}^{(2)}$ and $F_{4}^{(1)}$ a proper Young wall $(y_{k}(n_{k}))_{k=0}^{\infty} \in \Zcal(\lambda)$ satisfies condition (\ref{adjacent right block property}) on columns $k+1$ and $k$ whenever $H(y_{k+1}\otimes y_{k}) + |y_{k+1}|_{0} - |y_{k}|_{0} \not= -1$.
\end{Prop}

Furthermore, we can precisely describe which possible pairs of adjacent columns in a proper Young wall \textit{do} in fact fail condition (\ref{adjacent right block property}).

\begin{Prop}
    A proper Young wall $(y_{k}(n_{k}))_{k=0}^{\infty} \in \Zcal(\lambda)$ fails condition (\ref{adjacent right block property}) on columns $k+1$ and $k$ if and only if $(y_{k+1}(n_{k+1}),y_{k}(n_{k}))$ is of the form
    \begin{align*}
        (y_{\emptyset}(m),y_{\emptyset}(m)), \quad
        (y_{\emptyset}(m),y_{\theta}(m+1)), \quad
        (y_{-\theta}(m),y_{\emptyset}(m+1)), \quad
        (y_{-\theta}(m),y_{\theta}(m+2)).
    \end{align*}
\end{Prop}
\begin{proof}
    We see from the proof of Proposition \ref{Fock right block proposition} that if condition (\ref{adjacent right block property}) fails then $n_{k} = n_{k+1} + H(y_{k+1}\otimes y_{k})$ and $y_{k+1},y_{k}\in \lbrace y_{\emptyset}, y_{\pm\theta} \rbrace$, whereby a simple check completes the proof.
\end{proof}

We can also easily deduce the following.

\begin{Cor} \label{Fock built on ground-state wall}
    In types $E_{6}^{(2)}$ and $F_{4}^{(1)}$ every proper Young wall is built on top of the ground-state wall.
\end{Cor}
\begin{proof}
    This follows from Proposition \ref{Fock right block proposition} since a Young wall differs from the ground-state wall in finitely many blocks, and thus matches it in all columns sufficiently far to the left.
\end{proof}

Nevertheless, it is important to note that not every Young wall satisfying the weakened right block property is proper.

\begin{proof}[Proof of Proposition \ref{Fock right block proposition}]
Throughout this proof we shall refer only to $B$ and $H$ for ease of notation, but remark that exactly the same argument works with $B'$ and $H'$.

Viewing a Young column inside column $k+1$ of the Young wall pattern and mapping it to the Young column with blocks in the same positions but in column $k$ of the Young wall pattern corresponds to the automorphism $z : b(n) \mapsto b(n-1)$ of $\Baff$.

It suffices to show that if $\Haff(a(m)\otimes b(n)) < 2$ and $\Haff(a(m)\otimes b(n)) \not= 0$ then there is a directed path $z(a(m)) = a(m-1) \rightarrow\dots\rightarrow b(n)$, since going along an arrow in $\Baff$ corresponds to adding a block in the Young column model.

Without loss of generality we can take $m = 1$ and thus by Definition \ref{affine energy function} consider $n \geq H(a\otimes b)$ and $n \not= H(a\otimes b) + 1$.
The following lemma -- proved simply by inspecting the crystal graph of $B$ -- allows us to further restrict to the case $n = H(a\otimes b)$.

\begin{Lem} \label{path to itself lemma}
    \begin{enumerate}
        \item There is a path $b(n) \rightarrow\dots\rightarrow b(n+1)$ precisely when $b\not= \emptyset,\pm\theta$. \label{path plus 1}
        \item There is always a path $b(n) \rightarrow\dots\rightarrow b(n+2)$. \label{path plus 2}
        \item There is always a path $b(n) \rightarrow\dots\rightarrow b(n+3)$. \label{path plus 3}
    \end{enumerate}
\end{Lem}

So it remains to confirm the existence of paths $a(0) \rightarrow \dots \rightarrow b(H(a\otimes b))$ in $\Baff$, or equivalently of paths $a \rightarrow \dots \rightarrow b$ in $B$ whose number of $0$-arrows is $H(a\otimes b)$.

We can calculate all $\mathrm{Arr}^{0}(a,b)$ -- defined to be the \textit{minimum} number of $0$-arrows in a path $a \rightarrow \dots \rightarrow b$ in $B$ -- using SageMath \cite{SageMath}.
Since $B$ is isomorphic to the Kirillov-Reshetikhin crystal $B^{1,1}$, the following code outputs its list of edges.

\begin{lstlisting}
sage: K = crystals.kirillov_reshetikhin.LSPaths(['X',n,r],1)
sage: K.digraph().edges()
\end{lstlisting}

With a simple `find and replace' procedure we can turn this into a list {\tt E} of \textit{weighted} edges where $0$-arrows have weight $1$ and all other arrows have weight $1000$.
For technical reasons we must also replace any {\tt Lambda[j]} in {\tt E} with {\tt Lj}.
We then compute a list of minimal path weights between any two vertices in the associated weighted digraph.

\begin{lstlisting}
sage: var('L0 L1 L2 L3 L4')
sage: from sage.graphs.base.boost_graph import floyd_warshall_shortest_paths
sage: D = DiGraph(E,weighted=True)
sage: floyd_warshall_shortest_paths(D)
\end{lstlisting}

The final digit of each minimal weight is precisely $\mathrm{Arr}^{0}(a,b)$.

By comparing with the list of $H(a\otimes b)$ values calculated in Section \ref{Sec: Perfect crystal of E and F} we see that all
\begin{align*}
    0 \leq H(a\otimes b) - \mathrm{Arr}^{0}(a,b) \leq 2
\end{align*}
and hence by Lemma \ref{path to itself lemma} there is a path $a(0) \rightarrow \dots \rightarrow b(H(a\otimes b))$ in $\Baff$ whenever $a\not= \emptyset,\pm\theta$ or $b\not= \emptyset,\pm\theta$.
The remaining cases are easily verified by inspecting the crystal graph of $B$.
\end{proof}

\newpage

\appendix
\section{Values of the energy functions \texorpdfstring{$H$}{H} and \texorpdfstring{$H'$}{H'}}\label{appe:HandH'values}
\renewcommand{\arraystretch}{2.2}

\begin{table}[H]
\centering
\rotatebox{90}{
\begin{minipage}{0.9\textheight}
\resizebox{0.9\textheight}{!}{
\begin{tabular}{|c||*{27}{c|}}
\hline
$a\backslash b$ & $(2321)$ & $(1321)$ & $(1221)$ & $(1211)$ & $(1210)$ & $(1111)$ & $(1110)$ & $(0111)$ & $(1100)$ & $(0110)$ & $(1000)$ & $(0100)$ & $\hspace{12.47086pt}\emptyset\hspace{12.47086pt}$ & $\hspace{9.678865pt}r_{1}\hspace{9.678865pt}$ & $\hspace{9.678865pt}r_{2}\hspace{9.678865pt}$ & $\overline{(0100)}$ & $\overline{(1000)}$ & $\overline{(0110)}$ & $\overline{(1100)}$ & $\overline{(0111)}$ & $\overline{(1110)}$ & $\overline{(1111)}$ & $\overline{(1210)}$ & $\overline{(1211)}$ & $\overline{(1221)}$ & $\overline{(1321)}$ & $\overline{(2321)}$ \\ \hline
& & & & & & & & & & & & & & & & & & & & & & & & & & & \\[-27pt] \hline
$(2321)$ & $\cell{2}$ & $\cell{2}$ & $\cell{2}$ & $\cell{2}$ & $\cell{2}$ & $\cell{2}$ & $\cell{2}$ & $\cell{2}$ & $\cell{2}$ & $\cell{2}$ & $\cell{2}$ & $\cell{2}$ & $\cell{1}$ & $\cell{2}$ & $\cell{2}$ & $\cell{2}$ & $\cell{2}$ & $\cell{2}$ & $\cell{2}$ & $\cell{2}$ & $\cell{2}$ & $\cell{2}$ & $\cell{2}$ & $\cell{2}$ & $\cell{2}$ & $\cell{2}$ & $\cell{2}$ \\ \hline
$(1321)$ & $\cell{1}$ & $\cell{2}$ & $\cell{2}$ & $\cell{2}$ & $\cell{2}$ & $\cell{2}$ & $\cell{2}$ & $\cell{2}$ & $\cell{2}$ & $\cell{2}$ & $\cell{2}$ & $\cell{2}$ & $\cell{1}$ & $\cell{2}$ & $\cell{2}$ & $\cell{2}$ & $\cell{2}$ & $\cell{2}$ & $\cell{2}$ & $\cell{2}$ & $\cell{2}$ & $\cell{2}$ & $\cell{2}$ & $\cell{2}$ & $\cell{2}$ & $\cell{2}$ & $\cell{2}$ \\ \hline
$(1221)$ & $\cell{1}$ & $\cell{1}$ & $\cell{2}$ & $\cell{2}$ & $\cell{2}$ & $\cell{2}$ & $\cell{2}$ & $\cell{2}$ & $\cell{2}$ & $\cell{2}$ & $\cell{2}$ & $\cell{2}$ & $\cell{1}$ & $\cell{2}$ & $\cell{2}$ & $\cell{2}$ & $\cell{2}$ & $\cell{2}$ & $\cell{2}$ & $\cell{2}$ & $\cell{2}$ & $\cell{2}$ & $\cell{2}$ & $\cell{2}$ & $\cell{2}$ & $\cell{2}$ & $\cell{2}$ \\
\hline
\end{tabular}
}
\end{minipage}
}
\caption{The energy function $H(b\otimes a)$ in type $E_{6}^{(2)}$}
\label{E6(2) energy function table}
\end{table}
\renewcommand{\arraystretch}{2.2}
\begin{table}[H]
\centering
\rotatebox{90}{
\begin{minipage}{0.9\textheight}
\resizebox{0.9\textheight}{!}{
\begin{tabular}{|c||*{53}{c|}}
\hline
$a\backslash b$ & $(2342)$ & $(1342)$ & $(1242)$ & $(1232)$ & $(1231)$ & $(1222)$ & $(1221)$ & $(1122)$ & $(1220)$ & $(1121)$ & $(0122)$ & $(1120)$ & $(1111)$ & $(0121)$ & $(1110)$ & $(0120)$ & $(0111)$ & $(1100)$ & $(0110)$ & $(0011)$ & $(1000)$ & $(0100)$ & $(0010)$ & $(0001)$ & $\hspace{12.47086pt}\emptyset\hspace{12.47086pt}$ & $\hspace{9.678865pt}r_{1}\hspace{9.678865pt}$ & $\hspace{9.678865pt}r_{2}\hspace{9.678865pt}$ & $\hspace{9.678865pt}r_{3}\hspace{9.678865pt}$ & $\hspace{9.678865pt}r_{4}\hspace{9.678865pt}$ & $\overline{(0001)}$ & $\overline{(0010)}$ & $\overline{(0100)}$ & $\overline{(1000)}$ & $\overline{(0011)}$ & $\overline{(0110)}$ & $\overline{(1100)}$ & $\overline{(0111)}$ & $\overline{(0120)}$ & $\overline{(1110)}$ & $\overline{(0121)}$ & $\overline{(1111)}$ & $\overline{(1120)}$ & $\overline{(0122)}$ & $\overline{(1121)}$ & $\overline{(1220)}$ & $\overline{(1122)}$ & $\overline{(1221)}$ & $\overline{(1222)}$ & $\overline{(1231)}$ & $\overline{(1232)}$ & $\overline{(1242)}$ & $\overline{(1342)}$ & $\overline{(2342)}$ \\ \hline
& & & & & & & & & & & & & & & & & & & & & & & & & & & & & & & & & & & & & & & & & & & & & & & & & & & & & \\[-27pt]
\hline
$(2342)$ & $\cell{2}$ & $\cell{2}$ & $\cell{2}$ & $\cell{2}$ & $\cell{2}$ & $\cell{2}$ & $\cell{2}$ & $\cell{2}$ & $\cell{2}$ & $\cell{2}$ & $\cell{2}$ & $\cell{2}$ & $\cell{2}$ & $\cell{2}$ & $\cell{2}$ & $\cell{2}$ & $\cell{2}$ & $\cell{2}$ & $\cell{2}$ & $\cell{2}$ & $\cell{2}$ & $\cell{2}$ & $\cell{2}$ & $\cell{2}$ & $\cell{1}$ & $\cell{2}$ & $\cell{2}$ & $\cell{2}$ & $\cell{2}$ & $\cell{2}$ & $\cell{2}$ & $\cell{2}$ & $\cell{2}$ & $\cell{2}$ & $\cell{2}$ & $\cell{2}$ & $\cell{2}$ & $\cell{2}$ & $\cell{2}$ & $\cell{2}$ & $\cell{2}$ & $\cell{2}$ & $\cell{2}$ & $\cell{2}$ & $\cell{2}$ & $\cell{2}$ & $\cell{2}$ & $\cell{2}$ & $\cell{2}$ & $\cell{2}$ & $\cell{2}$ & $\cell{2}$ & $\cell{2}$ \\
\hline
\end{tabular}
}
\end{minipage}
}
\caption{The energy function $H'(b\otimes a)$ in type $F_{4}^{(1)}$}
\label{F4(1) energy function table}
\end{table}

\section{Crystal graphs of the Young column realizations for \texorpdfstring{$B$}{B} and \texorpdfstring{$B'$}{B'}}\label{realization of B and B'}

\input{Young_column_realization_E}
\input{Young_column_realization_F}

\section{Top parts of the crystals \texorpdfstring{$\mathcal Y(\Lambda_0)$}{Y(Lambda0)} and \texorpdfstring{$\mathcal Z(\Lambda_0)$}{Z(Lambda0)}}\label{top crystal}
\input{Top_crystal_E}
\input{Top_crystal_F}

\pagebreak


\begin{bibsection}
\begin{biblist}


\bib{BFKL06}{article}{
    title={Level 1 perfect crystals and path realizations of basic representations at q=0},
    author={G. Benkart},
    author={I. Frenkel},
    author={S.-J. Kang},
    author={H. Lee},
    journal={Int. Math. Res. Not.},
    volume={2006},
    number={10312},
    pages={1--28},
    year={2006},
    note={\url{https://doi.org/10.1155/IMRN/2006/10312}},
}

\bib{FHKS23}{article}{
    title={Young wall construction of level-1 highest weight crystals over $U_{q}(D_{4}^{(3)})$ and $U_{q}(G_{2}^{(1)})$},
    author={Z. Fan},
    author={S. Han},
    author={S.-J. Kang},
    author={Y.-S. Shin},
    journal={J. Algebra},
    year={2023},
    note={\url{https://doi.org/10.1016/j.jalgebra.2023.08.001}},
}

\bib{JM11}{article}{
    title={On Demazure crystals for $U_{q}(G_{2}^{(1)})$},
    author={R. L. Jayne},
    author={K. C. Misra},
    journal={Proc. Amer. Math. Soc.},
    volume={139},
    number={7},
    pages={2343--2356},
    year={2011},
    note={\url{https://doi.org/10.1090/S0002-9939-2010-10663-9}},
}

\bib{HKL04}{article}{
    title={Young wall realization of crystal graphs for $U_{q}(C_{n}^{(1)})$},
    author={J. Hong},
    author={S.-J. Kang},
    author={H. Lee},
    journal={Comm. Math. Phys.},
    volume={244},
    number={1},
    pages={111--131},
    year={2004},
    note={\url{https://doi.org/10.1007/s00220-003-0966-6}},
}

\bib{Kang03}{article}{
    title={Crystal bases for quantum affine algebras and combinatorics of Young walls},
    author={S.-J. Kang},
    journal={Proc. London Math. Soc.},
    volume={86},
    number={1},
    pages={29--69},
    year={2003},
    note={\url{https://doi.org/10.1112/S0024611502013734}},
}

\bib{KMN1}{article}{
    title={Affine crystals and vertex models},
    author={S.-J. Kang},
    author={M. Kashiwara},
    author={K. C. Misra},
    author={T. Miwa},
    author={T. Nakashima},
    author={A. Nakayashiki},
    journal={Int. J. Mod. Phys. A},
    volume={7},
    number={supp01a},
    pages={449--484},
    year={1992},
    note={\url{https://doi.org/10.1142/S0217751X92003896}},
}

\bib{KMN2}{article}{
    title={Perfect crystals of quantum affine Lie algebras},
    author={S.-J. Kang},
    author={M. Kashiwara},
    author={K. C. Misra},
    author={T. Miwa},
    author={T. Nakashima},
    author={A. Nakayashiki},
    journal={Duke Math. J.},
    volume={68},
    number={3},
    pages={499--607},
    year={1992},
    note={\url{https://doi.org/10.1215/S0012-7094-92-06821-9}},
}

\bib{KK03}{inproceedings}{
    title={Fock space representations for the quantum affine algebra $U_q(C_2^{(1)})$},
    author={S.-J. Kang},
    author={J.-H. Kwon},
    conference={
    title={Combinatorial and geometric representation theory},
    date={October 2001},
    address={Seoul National University, Seoul, Korea},
    },
    book={
    volume={325},
    publisher={Amer. Math. Soc.},
    series={Contemp. Math.},
    date={2003},
    },
    pages={109--131},
    note={\url{https://doi.org/10.1090/conm/325}},
}

\bib{KK04}{article}{
	title={Crystal bases of the Fock space representations and string functions},
	author={S.-J. Kang},
	author={J.-H. Kwon},
	journal={J. Algebra},
	volume={280},
	number={1},
	pages={313--349},
	year={2004},
	note={\url{https://doi.org/10.1016/J.JALGEBRA.2004.04.013}},
}

\bib{KK08}{article}{
    title={Fock space representations of quantum affine algebras and generalized Lascoux-Leclerc-Thibon algorithm},
    author={S.-J. Kang},
    author={J.-H. Kwon},
    journal={J. Korean Math. Soc.},
    volume={45},
    number={4},
    pages={1135--1202},
    year={2008},
    note={\url{https://doi.org/10.4134/JKMS.2008.45.4.1135}},
}

\bib{KL06}{article}{
    title={Higher level affine crystals and Young walls},
    author={S.-J. Kang},
    author={H. Lee},
    journal={Algebr. Represent. Theory},
    volume={9},
    number={6},
    pages={593--632},
    year={2006},
    note={\url{https://doi.org/10.1007/s10468-006-9013-6}},
}

\bib{KL09}{article}{
    title={Crystal bases for quantum affine algebras and Young walls},
    author={S.-J. Kang},
    author={H. Lee},
    journal={J. Algebra},
    volume={322},
    number={6},
    pages={1979--1999},
    year={2009},
    note={\url{https://doi.org/10.1016/j.jalgebra.2009.06.010}},
}

\bib{KM94}{article}{
	title={Crystal bases and tensor product decompositions of $U_q(G_2)$-module},
	author={S.-J. Kang},
	author={K.C.Misra},
	journal={J. Algebra},
	volume={163},
	number={3},
	pages={675-691},
	year={1994},
	note={\url{https://doi.org/10.1006/jabr.1994.1037}},
}

\bib{Kas90}{article}{
    title={Crystalizing the $q$-analogue of universal enveloping algebras},
    author={M. Kashiwara},
    journal={Comm. Math. Phys.},
    volume={133},
    number={2},
    pages={249--260},
    year={1990},
    note={\url{https://doi.org/10.1007/BF02097367}},
}

\bib{Kas91}{article}{
    title={On crystal bases of the $q$-analogue of universal enveloping algebras},
    author={M. Kashiwara},
    journal={Duke Math. J.},
    volume={63},
    number={2},
    pages={465--516},
    year={1991},
    note={\url{https://doi.org/10.1215/S0012-7094-91-06321-0}},
}

\bib{Kas02}{article}{
    title={On level zero representations of quantized affine algebras},
    author={M. Kashiwara},
    journal={Duke Math. J.},
    volume={112},
    number={1},
    pages={117--175},
    year={2002},
    note={\url{https://doi.org/10.1215/S0012-9074-02-11214-9}},
}

\bib{KMOY07}{article}{
    title={Perfect crystals for $U_q(D_4^{(3)})$},
    author={M. Kashiwara},
    author={K. C. Misra},
    author={M. Okado},
    author={D. Yamada},
    journal={J. Algebra},
    volume={317},
    number={1},
    pages={392--423},
    year={2007},
    note={\url{https://doi.org/10.1016/j.jalgebra.2007.02.021}},
}

\bib{KMPY96}{article}{
    title={Perfect crystals and q-deformed Fock spaces},
    author={M. Kashiwara},
    author={T. Miwa},
    author={J.-U. H. Petersen},
    author={C. M. Yung},
    journal={Sel. Math. New Ser.},
    volume={2},
    number={3},
    pages={415--499},
    year={1996},
    note={\url{https://doi.org/10.1007/BF01587950}},
}

\bib{KMS95}{article}{
    title={Decomposition of $q$-deformed Fock spaces},
    author={M. Kashiwara},
    author={T. Miwa},
    author={E. Stern},
    journal={Sel. Math. New Ser.},
    volume={1},
    pages={787--805},
    year={1995},
    note={\url{https://doi.org/10.1007/BF01587910}},
}

\bib{KN94}{article}{
    title={Crystal Graphs for Representations of the $q$-Analogue of Classical Lie Algebras},
    author={M. Kashiwara},
    author={T. Nakashima},
    journal={J. Algebra},
    volume={165},
    number={2},
    pages={295--345},
    year={1994},
    note={\url{https://doi.org/10.1006/jabr.1994.1114}},
}

\bib{KS04}{article}{
	title={Crystal bases and generalized Lascoux–Leclerc–Thibon (LLT) algorithm for the quantum affine algebra $U_q(C_n^{(1)})$},
	author={J.-A. Kim},
	author={D.-U. Shin},
	journal={J. Math. Phys},
	volume={45},
	number={12},
	pages={4878–4895},
	year={2004},
	note={\url{https://doi.org/10.1063/1.1811791}},
}

\bib{Lau23}{article}{
    title={Young wall realizations of level 1 irreducible highest weight and Fock space crystals of quantum affine algebras in type E},
    author={D. Laurie},
    journal={arXiv preprint},
    year={2023},
    eprint={arXiv:2311.03905},
}

\bib{Lus90}{article}{
    title={Canonical bases arising from quantized enveloping algebras},
    author={G. Lusztig},
    journal={J. Amer. Math. Soc.},
    volume={3},
    number={2},
    pages={447--498},
    year={1990},
    note={\url{https://doi.org/10.2307/1990961}},
}

\bib{Lus91}{article}{
	title={Quivers, perverse sheaves, and quantized enveloping algebras},
	author={G. Lusztig},
	journal={J. Amer. Math. Soc.},
	volume={4},
	number={2},
	pages={365--421},
	year={1991},
	note={\url{https://doi.org/10.2307/2939279}},
}

\bib{MM90}{article}{
	title={Crystal base for the basic representation of $U_q(\hat{\mathfrak{sl}}(n))$},
	author={K. C. Misra},
	author={T. Miwa},
	journal={Comm. Math. Phys.},
	volume={134},
	pages={79--88},
	year={1990},
	note={\url{https://doi.org/10.1007/BF02102090}},
}

\bib{MMO10}{article}{
    title={Zero action on perfect crystals for $U_q(G_2^{(1)})$},
    author={K. C. Misra},
    author={M. Mohamad},
    author={M. Okado},
    journal={SIGMA},
    volume={6},
    pages={022},
    year={2010},
    note={\url{http://dx.doi.org/10.3842/SIGMA.2010.022}},
}

\bib{Pre04}{article}{
	title={Fock space representations and crystal bases for  $C_n^{(1)}$},
	author={A. Premat},
	journal={J. Algebra},
	volume={278},
	number={1},
	pages={227--241},
	year={2004},
	note={\url{https://doi.org/10.1016/j.jalgebra.2004.01.028}},
}

\bib{SageMath}{manual}{
    author={Developers, The~Sage},
    title={{S}agemath, the {S}age {M}athematics {S}oftware {S}ystem ({V}ersion 9.8)},
    date={2023},
    note={\tt{https://www.sagemath.org}. \url{https://doi.org/10.5281/zenodo.593563}}
}

\bib{Ste95}{article}{
    title={Semi-infinite wedges and vertex operators},
    author={E. Stern},
    journal={Int. Math. Res. Not.},
    volume={1995},
    number={4},
    pages={201--220},
    year={1995},
    note={\url{https://doi.org/10.1155/S107379289500016X}},
}

\bib{Yam98}{article}{
    title={Perfect crystals of $U_q(G_2^{(1)})$},
    author={S. Yamane},
    journal={J. Algebra},
    volume={210},
    number={2},
    pages={440--486},
    year={1998},
    note={\url{https://doi.org/10.1006/jabr.1998.7597}},
}

\end{biblist}
\end{bibsection}
\end{document}